\documentclass{article}
\usepackage{amsmath,amsfonts,amssymb,amsthm,graphicx}
 \usepackage{dsfont}    
\usepackage{enumerate}
\usepackage{color}
\usepackage[numbers]{natbib}
\usepackage{multirow}
\usepackage{comment}
\usepackage[margin=1in]{geometry}

\usepackage{booktabs,subcaption,dcolumn}
\usepackage{diagbox}
\usepackage{pdflscape}
\usepackage{pgfplots}
\pgfplotsset{compat=1.17}
\usepgfplotslibrary{fillbetween}

\usepackage{algorithm}
\usepackage[noend]{algpseudocode}  
\usepackage[colorlinks=true,citecolor=blue,pdfpagemode=UseNone,pdfstartview=FitH]{hyperref}

\newif\ifFULL
\ifFULL

\fi

\newif\ifRuodu
\ifRuodu
  \renewcommand{\ge}{\geqslant}
  \renewcommand{\le}{\leqslant}
  \renewcommand{\epsilon}{\varepsilon}
  
\fi

\renewcommand{\d}{\mathrm{d}}

\newcommand{\p}{\mathbb{P}}

\newcommand{\E}{\mathbb{E}}    

\newcommand{\R}{\mathbb{R}}    
\newcommand{\N}{\mathbb{N}}    

\def\lawis{\buildrel \mathrm{d} \over \sim}

\newcommand{\id}{\mathds{1}}

\newcommand{\X}{\mathcal{X}}

\theoremstyle{plain}
\newtheorem{theorem}{Theorem} 
\newtheorem{corollary}{Corollary}
\newtheorem{lemma}{Lemma}
\newtheorem{proposition}{Proposition}

\theoremstyle{definition}

\theoremstyle{remark}
\newtheorem{remark}{Remark}

\renewcommand{\cite}{\citet}  

\title{Improved thresholds for e-values} 
\author{Christopher Blier-Wong\thanks{Department of Statistical Sciences,
  University of Toronto,
   Canada. \href{mailto:christopher.blierwong@utoronto.ca}{christopher.blierwong@utoronto.ca}.}
\and
  Ruodu Wang\thanks%
  {Department of Statistics and Actuarial Science,
  University of Waterloo,
   Canada. \href{mailto:wang@uwaterloo.ca}{wang@uwaterloo.ca}.}}

\begin{document}
\maketitle
 The rejection threshold used for e-values and e-processes is by default set to $1/\alpha$ for a guaranteed type-I error control at $\alpha$, based on Markov's and Ville's inequalities. This threshold can be wasteful in practical applications. We discuss how this threshold can be improved under additional distributional assumptions on the e-values; some of these assumptions are naturally plausible and empirically observable, without knowing explicitly the form or model of the e-values. For small values of $\alpha$, the threshold can roughly be improved (divided) by a factor of $2$ for decreasing or unimodal densities, and by a factor of $e$ for decreasing or unimodal-symmetric densities of log-transformed e-values. Moreover, we propose to use the supremum of comonotonic e-values, which is shown to preserve the type-I error guarantee. We also propose some preliminary methods to boost e-values in the e-BH procedure under some distributional assumptions while controlling the false discovery rate. Through a series of simulation studies, we demonstrate the effectiveness of our proposed methods in various testing scenarios, showing enhanced power.

\begin{abstract}

\end{abstract}

\section{Introduction}

In statistical hypothesis testing, p-values and e-values summarize the evidence from a test against a null hypothesis. E-values have recently gained attention as a promising alternative to p-values, with inherent advantages over methods based on p-values, including validity in optional sampling and stopping, the ability to provide an intuitive notion of statistical significance, robustness to model misspecification, and flexibility within multiple testing procedures under dependence.
We refer to \cite{WRB20,S21,VW22,WR22} and \cite{GDK24} for advantages of e-values, and to \cite{RGVS23, ramdas2025hypothesisb} for recent developments.


Classic testing methods often use a p-value with a significance level $\alpha \in (0, 1)$, representing the allowed type-I error. A statistician often starts with a test statistic and designs a rejection region such that the probability of that test statistic belonging to the rejection region, under the null hypothesis, is at most $\alpha$. A simple analogous rejection region for e-values is $[1/\alpha, \infty)$ to guarantee a type-I error control of $\alpha$; the lower bound of the rejection region is based on Markov's inequality. We will refer to an e-test as a test that rejects when an e-value exceeds a certain threshold. Using a threshold of $1 / \alpha$ for e-tests is often wasteful in practice, as it is typically too conservative, resulting in a type-I error that is much smaller than the desired level $\alpha$, and thus a loss of power. By selectively using information about the distribution or dependence structure of e-values, we can increase the power of e-tests while maintaining type-I error control, striking a balance between safety and efficiency. 


The paper's main objective is to provide valid thresholds for e-values, under some distributional assumptions, that control the type-I error of e-tests. Including some distributional assumptions enables one to construct larger rejection regions (lower rejection thresholds), improving the power of e-tests while retaining some desirable properties and putting e-values on a more comparable scale to p-values. 

If one knows the true distribution of the e-variable under the null, one may design an e-test to obtain exact type-I error control by inverting the cumulative distribution function of the e-variable; this would be equivalent to an approach based on p-values. However, some e-variables are constructed in a complicated way, where the distribution of the e-variable is challenging to study; yet, we may be confident about specific features. We provide two examples below. In both cases, 
their detailed construction is described in Appendix \ref{app:constr-e-intro}.

Suppose that we have a high-dimensional dataset and wish to identify informative variables for a target response through variable selection. A flexible approach to accomplish this task is the model-X knockoffs (see \cite{barber2015controlling} and \cite{candes2018panning}), which control the false discovery rate (FDR). There are many ways to construct knockoffs, but identifying the resulting distribution of the knockoffs is not straightforward. \cite{RB24} showed that knockoffs are e-values in a slightly more general sense, called compound e-values by \cite{IWR24}. We may improve the threshold of e-tests on knockoffs with distributional information on the shape of their distribution. In Figure \ref{fig:knockoff-e-values}, we present a histogram of 10,000
knockoffs from linear regression when the number of variables is 80 and the number of non-null variables is 0 (left) and 8 (right). In both cases, the knockoffs appear to have a single mode, implying that the e-values also have a single mode. 


\begin{figure}
    \centering

    \includegraphics[width = \textwidth]{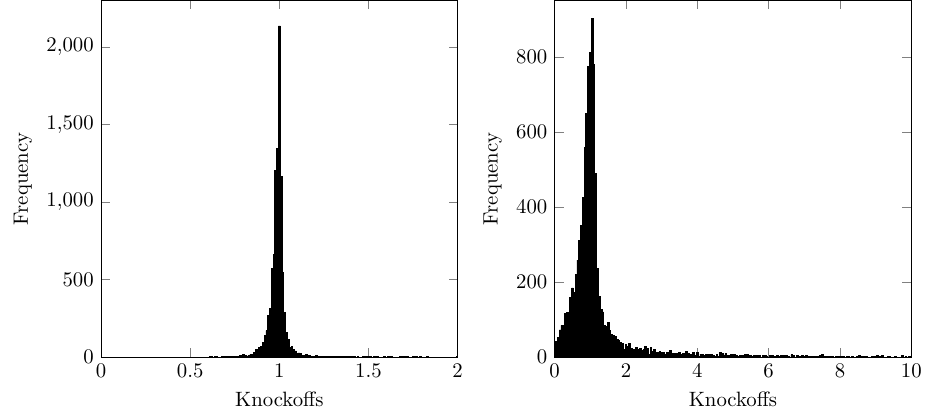}
    \caption{Histogram of null knockoffs, which are e-values up to a constant. The left is based on knockoffs where all the nulls in the dataset are true; the right is based on knockoffs where 10\% of variables are non-null.}
    \label{fig:knockoff-e-values}
\end{figure}

Universal inference, proposed by \cite{WRB20}, is a powerful method for developing hypothesis tests that provide finite-sample guarantees on the type-I error. This technique is especially useful for situations where classical p-tests are not feasible, such as when dealing with composite null hypotheses in irregular statistical models, provided that a maximum likelihood estimator can be computed. Assume we have iid observations $X_1, \dots, X_n \sim Q$ and want to test the composite null $Q = {\rm N}(0, 1)$ against the mixture $Q = 0.5{\rm N}(\mu_1, 1) + 0.5{\rm N}(\mu_2, 1)$ for unknown $(\mu_1, \mu_2)$. A split likelihood-ratio test statistic may be used to construct a valid e-test; see \cite{WRB20}. In Figure \ref{fig:ui-e-values}, we present a histogram of 10,000 e-values and log-transformed e-values. Even if the shape of the histogram of e-values seems complicated, we can observe that its density is decreasing for e-values larger than one, and the density of the log-transformed e-values is decreasing for log e-values larger than zero. With this distributional knowledge, we can improve the e-test's threshold without knowing the true distribution of the null e-variable. \cite{tse2022note} also proposed a method to improve the threshold of universal inference tests based on the asymptotic distribution of universal inference e-values, but the finite-sample type-I error control is not guaranteed theoretically. 

\begin{figure}
    \centering
    \includegraphics[width = \textwidth]{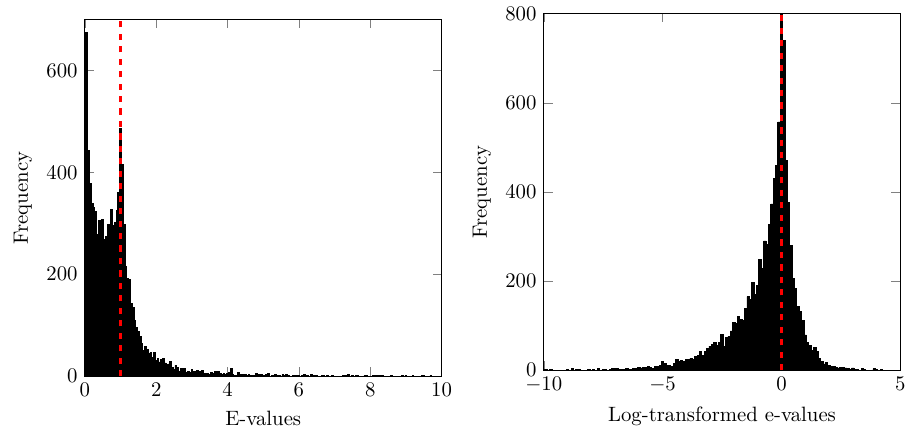}
    \caption{Histogram of universal inference e-values and log e-values. The red line is at $1$ on the left plot and $0$ on the right plot, and the density decreases after the red lines.}
    \label{fig:ui-e-values}
\end{figure}

The rest of the paper is dedicated to obtaining improved thresholds that can be useful in situations like the two examples above and many others. 
Section \ref{sec:2} provides preliminaries on e-values and their thresholds, as well as some basic results.
Section \ref{sec:3} contains our main results on the improved thresholds for e-values under various conditions, including decreasing and unimodal densities, as well as conditions on densities or random variables after a log-transformation. The most useful results reveal that, for small values of $\alpha$, the threshold can roughly be improved (divided) by a factor of $2$ for decreasing (only required on $[1,\infty)$) or unimodal densities,
and by a factor of $e=\exp\{1\}$ for decreasing (only required on $[0,\infty)$) or unimodal-symmetric densities of the log-transformed e-variable. Section \ref{sec:perturbation} tells us that, if an improved threshold applies to one set of distributions, then it also applies approximately to another set that is close to it.

In Section \ref{sec:como}, we propose tests with valid type-I errors based on taking the supremum over a set of e-variables. 
Section \ref{sec:combined} provides some improved thresholds for stopped e-processes, under the assumption that the stopping time is predictable.
Section \ref{sec:boosting} discusses methods to boost e-values in the e-BH procedure for multiple hypothesis testing with distributional assumptions.
We acknowledge that the results  in Sections \ref{sec:combined} and \ref{sec:boosting} are restrictive as they require strong assumptions; they primarily serve as an exploration of possible improvements in stopped e-processes and multiple testing, a topic of great importance and challenge.
Section \ref{sec:simulation} presents simulation studies that demonstrate the enhanced power achieved by using the improved thresholds. Section \ref{sec:concl} concludes.
We present a few intuitive proofs in the main text, while additional proofs and other results are included in the appendices.

\section{Thresholds for e-values: preliminaries}
\label{sec:2}

\subsection{E-variables and thresholds for e-tests}
We begin by describing the basic setting and terminology of our paper. 
 Our terminology follows  \cite{VW21}. Throughout, fix a sample space.
  A hypothesis is a collection $H$ of probability measures on the sample space.  
An \emph{e-variable} $E$ for a hypothesis $H$ is a $[0,\infty]$-valued random variable satisfying $\E^Q[E]\le1$ for all $Q\in H$. For the general background on e-values in hypothesis testing, see \cite{S21}, \cite{ramdas2025hypothesisb}, \cite{VW21} and \cite{GDK24}.
E-variables are often obtained from stopping an \emph{e-process} $(E_t)_{t \geq 0}$, which is a nonnegative stochastic process adapted to a pre-specified filtration such that $\mathbb{E}^Q[E_\tau] \leq 1$ for any stopping time $\tau$ and any $Q\in H$. 
Realizations of 
e-variables are referred to as 
e-values. 
 
 We will also encounter p-values and p-variables. A \emph{p-variable} $P$ for a hypothesis $H$ is a  random variable  that satisfies $Q(P\le \alpha)\le \alpha$ for all $\alpha \in (0,1)$ and all $Q\in H$.  

We fix a probability measure $\p$. All e-variables and p-variables are for the hypothesis $\{\p\}$, and the expectation $\E$ is with respect to $\p$. To study properties and thresholds of e-variables and p-variables, such an assumption is without loss of generality; see the explanation in \cite{VW21}. 
Denote by $\mathcal E_0$ the set of all e-variables.
All terms such as ``increasing'' and ``decreasing'' are in the non-strict sense.
We write $a\wedge b=\min\{a,b\}$ and $a\vee b=\max\{a,b\}$ for real numbers $a,b$. 

Below, we describe the main quantity of interest, the worst-case type-I error for a given set $\mathcal E$ of e-variables. 
The set $\mathcal E$ will be assumed to satisfy some distributional conditions to be specified later.  
Our goal is to find  the quantity  $R_{\gamma}(\mathcal E)$ for $\gamma >0 $, defined by
$$
R_{\gamma}(\mathcal E)= \sup_{E\in \mathcal E} \p(E\ge 1/\gamma ),
$$  
that is, the largest probability that $\p(E\ge 1/\gamma )$ can attain for $E\in \mathcal E$.
Markov's inequality gives 
$$
 \p(E\ge 1/\alpha) \le \alpha \mbox{~~~for all $\alpha \in (0,1]$ and all e-variables $E$},
 $$ 
 and moreover, the inequality is attainable. Hence, for any set $\mathcal E$ of e-variables, it holds that $R_{\gamma}(\mathcal E) \le  R_{\gamma}(\mathcal E_0)=\gamma$ for  $\gamma\in (0,1]$. 

 We are interested only in the case  $\gamma\in (0,1]$. For $\gamma>1$,
 it is usually the case that $R_{\gamma}(\mathcal E)=1$ because for most classes $\mathcal E$ that we consider, either $1\in \mathcal E$ or $1$ is the limit of elements of $\mathcal E$. Further, the usual interpretation of e-values does not permit rejection of the null hypothesis with a threshold less than $1$.

Markov's inequality is often seen as being quite conservative in statistical practice, 
because for an e-variable $E$, $ \p(E\ge 1/\alpha) = \alpha $  if and only if $E$ has a two-point distribution such that $\p(E=1/\alpha)=\alpha = 1-\p(E=0)$. 
Such an e-variable is unlikely to be observed in practice unless it is obtained by an algorithm which stops an e-process when it reaches $1/\alpha$.
This suggests that there is room to improve $R_{\gamma}(\mathcal E)$ under additional conditions. 

To run an e-test at level $\alpha \in (0,1)$ for the e-variable $E$,   
one needs to find a threshold $t>0$,   the smaller the better, such that 
 $
 \p(E \ge t) \le \alpha.
  $
  For this, we intuitively should use  the smallest $t$ such that 
 $R_{1/t}(\mathcal E)\le \alpha$, which satisfies  $R_{1/t}(\mathcal E)=\alpha$
in case   $\gamma\mapsto R_{\gamma}(\mathcal E)$  is continuous.
We only consider $t\ge 1$, corresponding to $\gamma\le 1$.

 \subsection{Basic properties}
\label{sec:R1-1}

The following result gives an equivalent formulation of the smallest threshold  $t$   for a set of e-variables.  Denote by $q_\beta(X)$ the left $(1-\beta)$-quantile  function of $X$, that is, 
 $$q_\beta(X) =\inf \{x\in \R: \p(X\le x)\ge 1-\beta\} \mbox{~~~for $\beta \in (0,1)$}.$$
 The intuition is that the threshold for a set of e-variables is the supremum of the corresponding quantiles, just like the threshold for one test statistic, which should be chosen as its quantile at a given level.
 \begin{lemma}\label{lem:1}
 For $\alpha \in (0,1)$, the quantity
 $ 
T_{\alpha}(\mathcal E) : = \inf\{t\ge 1: R_{1/t}(\mathcal E) \le \alpha\}
 $ satisfies 
$$T_{\alpha}(\mathcal E)= \left( \sup_{E\in \mathcal E} q_{\alpha}(E)\right)\vee 1.$$
If $\gamma\mapsto R_{\gamma}(\mathcal E)$  is continuous, then 
 $T_{\alpha}(\mathcal E)$
 is the smallest real number $t\ge 1$ such that  $ \p(E \ge t) \le \alpha$ for all $E\in \mathcal E$.  
 \end{lemma}
   \begin{proof} 
   Since $ R_{\gamma}(\mathcal E)\le \gamma$ for $\gamma \in (0,1]$, 
   we have  that $\{t\ge 1: R_{1/t}(\mathcal E) \le \alpha\}$ is not empty.
Note that 
\begin{align*}
T_\alpha(\mathcal E)  &  \ge \inf \{ t\ge 1 :   \p(E > t) \le \alpha \mbox{ for all $E\in \mathcal E$} \}  
 = \left( \sup_{E\in \mathcal E} q_{\alpha}(E)\right)\vee 1.
\end{align*}
Take any $\epsilon\in (0,1)$. We have
 \begin{align*}
T_{\alpha}(\mathcal E) -\epsilon  &  =  \inf \{t\ge 1-\epsilon : R_{1/(t+\epsilon)}(\mathcal E) \le \alpha\}   
\\& = \inf \left\{ t\ge 1 -\epsilon :    \sup_{E\in \mathcal E} \p(E \ge  t +\epsilon) \le \alpha  \right\}  
\\& = \inf   \{ t \ge 1  -\epsilon :     \p(E \ge  t+\epsilon) \le \alpha \mbox{ for all $E\in \mathcal E$} \}  
 \\&\le \inf \{ t\ge 1  -\epsilon :   \p(E > t ) \le \alpha \mbox{ for all $E\in \mathcal E$} \}  
\\ &  = \left(\sup_{E\in \mathcal E} q_{\alpha}(E)\right)\vee (1-\epsilon) \le  \left( \sup_{E\in \mathcal E} q_{\alpha}(E)\right)\vee 1\le T_\alpha(\mathcal E) .
\end{align*}
Since $\epsilon\in (0,1)$ is arbitrary, we have $T_\alpha(\mathcal E) =\left( \sup_{E\in \mathcal E} q_{\alpha}(E)\right)\vee 1$, showing  the first statement. 
If  $\gamma\mapsto R_{\gamma}(\mathcal E)$  is continuous, then 
 \begin{align*}
T_{\alpha}(\mathcal E)  &  = \min\{t\ge 1 : R_{1/t}(\mathcal E) \le \alpha\} 
  = \min \{ t \ge 1 :     \p(E \ge  t) \le \alpha \mbox{ for all $E\in \mathcal E$} \},
\end{align*}  
showing the second statement.
 \end{proof}
  \begin{remark}
  In the last statement of Lemma \ref{lem:1}, the condition that $\gamma\mapsto R_{\gamma}(\mathcal E)$ is continuous is essential. In general, there may not exist a smallest $t\ge 1$ such that  $ \p(E \ge t) \le \alpha$ for all $E\in \mathcal E$; this is the case even when $\mathcal E$ is a singleton, such as $\mathcal E=\{1\}$, for which we can see that $T_\alpha(\mathcal E)=1$ but $ \p(E \ge  T_\alpha(\mathcal E) )   =1$ for each $\alpha \in(0,1)$.
  \end{remark}

 \cite{VW21}
  studied e-to-p calibrators, which are mappings that convert any e-value into a p-value.
  Our function $\gamma\mapsto R_{1/\gamma}(\mathcal E)$  refines this concept.  
For a subset $\mathcal E$
  of e-variables, we say that a function $f:[0,\infty]\to[0,\infty)$ is an e-to-p calibrator on $\mathcal E$ if $f$ is decreasing and 
  $f(E)$ is a p-variable for all $E\in\mathcal E$.
  \cite{VW21} showed that 
  $x\mapsto (1/x)\wedge 1$
  is the only useful e-to-p 
 calibrator on $\mathcal E_0$ in the sense of being the smallest. 
  For various subsets  $\mathcal E$ of  $\mathcal E_0$, we can find better e-to-p 
 calibrators based on $R_{1/\gamma}(\mathcal E)$ than $x\mapsto (1/x)\wedge 1$.
  Moreover,
  the following result implies that any class $\mathcal E$ admits a smallest e-to-p 
 calibrator. This is in sharp contrast to the set of calibrators from p-values to e-values, which does not admit a smallest element, as well discussed by \cite{S21} and \cite{VW21}. 

\begin{theorem}\label{th:calibrator}
    The function $x \mapsto R_{1/x }\left(\mathcal{E}\right)$ on $[0,\infty]$ is an e-to-p  calibrator on $\mathcal{E}$, and it is the smallest such calibrator. 
\end{theorem}
\begin{proof}
    We first show that the function $g:x \mapsto R_{1/x}\left(\mathcal{E}\right) $ is an e-to-p calibrator on $\mathcal E$.
    This means
    $
    \p(g(E)\le \alpha) \le\alpha 
    $
    for all $\alpha \in (0,1)$ and $E\in \mathcal E$.
    By definition, for $x\in [0,\infty]$ and $E\in\mathcal E$, 
    \begin{align*}
        g( x) 
       \ge  \p(E\ge x ) 
    = \p(-E\le -x
        ) =  F_{-E}(-x)
        ,
    \end{align*}
where $F_{X}$ is the cdf of a random variable $X$. Since $F_{X}(X)$ is a p-variable for any random variable $X$, we have 
$\p(F_{-E}(-E)\le \alpha )\le \alpha$ for all $E\in \mathcal E$. 
Therefore, $\p(g(E)\le \alpha) \le \p(F_{-E}(-E)\le\alpha)\le \alpha$ for all $\alpha \in (0,1)$. 

    We now prove that $g$ is the smallest calibrator on $\mathcal{E}$. Suppose 
 that $f$ is another calibrator on $\mathcal{E}$ such that  $f(x)< g(x)$ for at least one $x\in [1,\infty]$.  By definition, there exists   $E\in \mathcal E$
 such that  $
        f(x) < \p(E \geq x).
 $
    This implies  $\p(f(E) \leq f(x)) \geq \p(E \geq x)> f(x),$ which violates the definition of $f(E)$ as a p-variable.
\end{proof}

Theorem \ref{th:calibrator} implies that by computing $R_{\gamma}(\mathcal E)$ or an upper bound on it, we can find calibrators to convert e-values  realized by elements of $\mathcal E$ into p-values, which work better than the e-to-p 
 calibrator $x\mapsto (1/x)\wedge 1$ on $\mathcal E_0$. This can be useful in procedures that take p-values as input, in e.g., that of \cite{BH95}.

We make a useful observation on the probability bound  $R_\gamma(\mathcal E)$ and the threshold $T_\alpha(\mathcal E)$ regarding mixture distributions, which will be useful for our later results on e-processes. 

\begin{proposition}
\label{prop:R1-1}
For a class $(E_\theta)_{\theta\in \Theta}$ of e-variables in $\mathcal E$ and a probability measure $\pi$ on $\Theta$, 
let $E_\pi$ follow the mixture  distribution, that is, 
$\p(E_\pi\le x) = \int_\Theta \p(E_\theta\le x) \pi (\d \theta)$ for $x\in \R$.  
Then $\p(E_\pi \ge 1/\gamma) \le R_\gamma(\mathcal E)$. 
\end{proposition}
In other words, the probability bounds and thresholds obtained for $\mathcal E$ also apply to the convex hull of $\mathcal E$ with respect to distribution mixtures.

 \section{Improved thresholds   under distributional information}
 \label{sec:3}

In this section, we study the improved e-thresholds under some conditions on the density function or distribution function of the e-variable by computing $R_{\gamma}(\mathcal E)$ or an upper bound on it.

\subsection{Decreasing and unimodal densities}
 
We first consider e-variables 
that have a decreasing density on their support, a decreasing density for values above 1, or a unimodal density. Denote by 
\begin{align*}
\mathcal E_{\rm D} &=\{E\in\mathcal E_0: \mbox{$E$  has a decreasing density on its support}\},\\    
\mathcal E_{\rm D>1} &=\{E\in\mathcal E_0: \mbox{$E$ has a decreasing density over $[1, \infty)$}\},\\
\mathcal E_{\rm U} &=\{E\in\mathcal E_0: \mbox{$E$  has a unimodal density on $[0,\infty)$}\}.
\end{align*}

The set $\mathcal E_{\rm D}$ contains, for instance,  exponential and Pareto distributions. 
Whenever we consider a density, we include its limit case; that is, we allow $E\in \mathcal E_{\rm D}$ to have a point mass at the left endpoint of its support. This convention does not change any results but simplifies many arguments in the proofs. 

The set $\mathcal E_{\rm D>1}$ is the set of e-variables with a decreasing density over $[1, \infty)$, but an arbitrary density over $[0, 1)$. Such e-variables are relevant to e-testing since e-values between 0 and 1 provide no evidence against the null; e-values above 1 are the ones that cause type-I errors, as they may appear in the rejection region of an e-test.

A random variable on $[0,\infty)$ is said to have a unimodal density, or a unimodal distribution, if it has a density function 
that is increasing on $[0,x]$ and decreasing on $[x,\infty)$ for some constant $x\in [0,\infty)$, 
and it can have a point mass at $x$. Such $x$ is called a mode of the random variable. 
The set of all modes of $E\in\mathcal E_{\rm U}$ may be a singleton or an interval.
The set $\mathcal E_{\rm U}$ contains, for instance, log-normal and gamma distributions.
Moreover, by Proposition \ref{prop:R1-1}, our obtained bounds also apply to the mixtures of the distributions within each class.
Note that $\mathcal E_{\rm D}\subseteq \mathcal E_{\rm D>1} $ and $\mathcal E_{\rm D}\subseteq \mathcal E_{\rm U}$
and hence $R_\gamma (\mathcal E_{\rm U}   ) \ge  R_\gamma (\mathcal E_{\rm D}   )$ and $R_\gamma (\mathcal E_{\rm D>1}   ) \ge  R_\gamma (\mathcal E_{\rm D}   )$ for all $\gamma \in (0,1]$. 

The proof techniques of the following result are similar to those in \cite{LSWY18}, \cite{BKV20}  and \cite{W24} for computing bounds on probabilities with given distributional information. \cite{H19} has different results on halving the bound in Markov's inequality via smoothing.

\begin{theorem}\label{th:dec-uni}
For $\gamma \in (0,1]$,
\begin{enumerate}
\item [(i)]
$R_\gamma ({\mathcal E_{\rm D}}  ) =\gamma /2$ if $\gamma \ne 1$ and $R_1 ({\mathcal E_{\rm D}}  ) = 1$;
\item [(ii)]    
    $R_\gamma(\mathcal{E}_{\rm D>1}) =  {\gamma }/({1+ \sqrt{1-\gamma^2}})$; 
    \item[(iii)]  
$R_\gamma (\mathcal E_{\rm U}   ) =(\gamma /2)\vee (2\gamma -1)$.
\end{enumerate}
\end{theorem} 

\begin{proof}


We only present the proof of part (i) here, which is the simplest, but 
it illustrates the basic techniques of proving more complicated statements. The proof of parts (ii) and (iii) are relegated to the appendices. 
We first show $\p(E\ge 1/\gamma )\le \gamma/2$ for each $E\in \mathcal E_{\rm D}$ and $\gamma \in (0, 1)$.
Let $f$ be the density function of $E$ and let $a$ be the left end-point of the support of $E$. For a fixed $z > 1$, set $b := z + \p(E \ge z)/f(z)$.
We construct an e-variable $E_0$ with density $g_0$ such that 
(a) $g_0(x) = f(z)$ for $a < x < b$;
(b) $E_0$ has a point mass at $a$ with probability 
$1 - \int_a^\infty g_0(x) \d x$.

\begin{figure}[ht]
    \centering
    \includegraphics{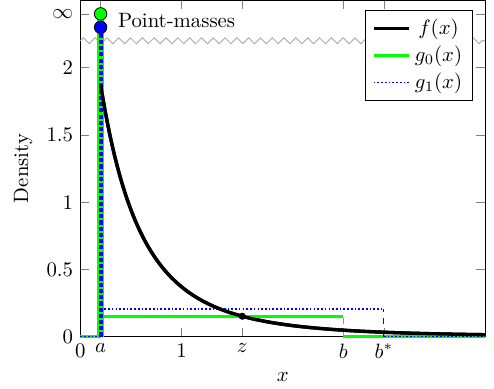}
    \caption{Proof sketch in part (i) of Theorem \ref{th:dec-uni}. 
    Here, $f$ and $g_0$ have the same area (probability) exceeding $z$, with $g_0$ having a smaller mean; $g_1$ has mean 1 and larger probability of exceeding $z$ than $g_0$.}
    \label{fig:pf-sketch-decreasing}
\end{figure}

We illustrate in Figure \ref{fig:pf-sketch-decreasing} how to construct $g_0$ from $f$.
Since $f$ has a decreasing density over $[a, \infty)$, we will have $f(x) \ge g_0(x)$ for $x \in [a, z)$ and $f(x) \le g_0(x)$ for $x > z$.
We can construct $g_0$ by shifting the area between $f(x)$ and $f(z)$, for $x \in (a, z)$ to a point mass at $a$ and shifting the area between $f(x)$ and $0$, for $x > b$ to the area between $f(x)$ and $f(z)$, for $x \in (z, b)$. 
Note that $b$ is chosen so that $\p(E_0 \in [z, b)) = \p(E \ge z)$, 
so $\p(E_0 \ge z) = \p(E \ge z)$. 
Further, since we construct $g_0$ by shifting the density of $f$ to the left, we have $\E[E_0] \le \E[E]$.  

Therefore, for any $E \in \mathcal{E}_{\rm D}$ with left end-point of support $a$, there exists an e-variable $E' \in \mathcal{E}_{\rm D}$ such that $E'$ has a point mass at $a$ and a uniform density on $[a, b]$ for some constant $b$, and $\p(E'\ge z) = \p(E \ge z)$. To show $\p(E \ge z) \leq 1/(2z)$, it suffices to look at the collection of e-variables that have a point mass at $a$ and a uniform density on $[a, b]$ for any $b > 1$.
Moreover, it suffices to consider the case $\E[E']=1$.

Note that $\E[E'] = (1 - w)a + w (a + b)/2$ for some $w\in [0, 1]$ and $b \geq 1$. The requirement $\E[E'] = 1$ implies that $w = 2(1 - a)/(b - a)$. Now, we seek $b$ that maximizes $\p(E'\ge z)$. 
We have 
$\p(E' \ge z) = w (b-z)/(b-a) = 2(1-a)(b-z)/(b-a)^2$. Taking the derivative with respect to $b$ and setting the result equal to zero, we find that $b^* := 2z - a$ for $z > 1$. Substituting this value of $b^*$ back into $\p(E'\ge z)$, we find that $\p(E'\ge z) = (1-a)/(2(z-a))$.  
We illustrate the density of $E'$ that maximizes $\p(E'\ge z)$, denoted by $g_1$, in Figure \ref{fig:pf-sketch-decreasing}.
Since $z > a$, the latter probability is a decreasing function of $a$, hence its supremum occurs for $a = 0$ and we have $\p(E\ge z) \le 1/(2z)$. 

It remains to show $ \p(E\ge 1/\gamma ) = \gamma/2$  for some $E\in \mathcal E_{\rm D}$. Take $E$ following a mixture distribution of a uniform distribution on $[0,2/\gamma]$ with weight $\gamma$ and a point mass at $0$. It is easy to see that $\p(E\ge 1/\gamma) =\gamma/2$ and $\E[E]=1$, showing the desired result. The last statement of part (i) is trivial by checking with $E=1$.  
\end{proof}

We obtain the following result by combining Theorems \ref{th:calibrator} and \ref{th:dec-uni}.

\begin{corollary}\label{coro:1}
If $E\in \mathcal E_{\rm D}$ then $ P:=(2E)^{-1} \id_{\{E> 1\}}  + \id_{\{E\le 1\}}$ is a p-variable.
  \end{corollary}
  
 The p-variable $P$ in Corollary \ref{coro:1} is not a continuous function of $E$ at $E=1$. This is because the quantity $R_{\gamma}(\mathcal E_{\rm D})$ is not continuous at $\gamma=1$. Indeed, $R_1 (\mathcal E_{\rm D})=1$ and $R_\gamma (\mathcal E_{\rm D})<1/2$ for $\gamma <1$.

\begin{remark} \label{rem:dec}
Proposition 6.4 of \cite{W24}
gives that $2E^{-1}$ is a p-variable 
for any e-variable $E$ with a decreasing density on $[0,\infty)$, using the concept of p*-values.
Contrasted to Corollary \ref{coro:1}, the 
separate treatment of $E>1$ and $E\le 1$ is absent in that result. This is
because if the left end-point of the support of $E$ is $0$, then $\p(E> 1/x)\le x/2$ for all $x>0$; in part (i) of Theorem \ref{th:dec-uni} we only have this probability bound for $x>1$ due to our less restrictive assumption on the support of $E$.
\end{remark}

The main message from Theorem \ref{th:dec-uni} is that when an e-variable has a decreasing density, its threshold can be boosted by a factor of $2$. That is, $T_\alpha(\mathcal{E}_{\rm D}) = 1/(2\alpha)$. 
Further, $T_\alpha(\mathcal{E}_{\rm D>1}) = 1/(2\alpha) + \alpha/2 = T_\alpha(\mathcal{E}_{\rm D}) + \alpha/2$.
Therefore, for small values of $\alpha$, we can boost the threshold by a factor slightly less than 2 when $E \in \mathcal{E}_{\rm D>1}$. This validates our message that, for small $\alpha$, the distribution of uninformative e-values (between 0 and 1) has little impact on the rejection region; the shape of the distribution of large e-values has the most influence on the rejection threshold. 
In part (iii) of Theorem \ref{th:dec-uni}, 
$R_\gamma (\mathcal E_{\rm U}   ) = \gamma /2 $,
for $\gamma \in (0,2/3]$   and $R_\gamma (\mathcal E_{\rm U}   ) = 2\gamma -1$ for $\gamma \in (2/3,1]$.
Since $\gamma\le 2/3$ is the most practical situation (meaning a type-I error control of $1/3$), the main message from part (iii) is similar to that from part (i): when an e-variable has a unimodal density, its threshold can be boosted by a factor of $2$ in practice. 
We provide in Figure \ref{fig:threshold-comparison-1} a comparison of different worst-case type-I errors and improved thresholds for $\mathcal{E}_{\rm 0}, \mathcal{E}_{\rm D},  \mathcal{E}_{\rm D>1}$ and $\mathcal{E}_{\rm U}$.

In Appendix \ref{sup:moment-constraints}, we provide some other results related to this section that further incorporate variance constraints on the e-variables.

\begin{figure}[ht]
    \centering

    \includegraphics[width = \textwidth]{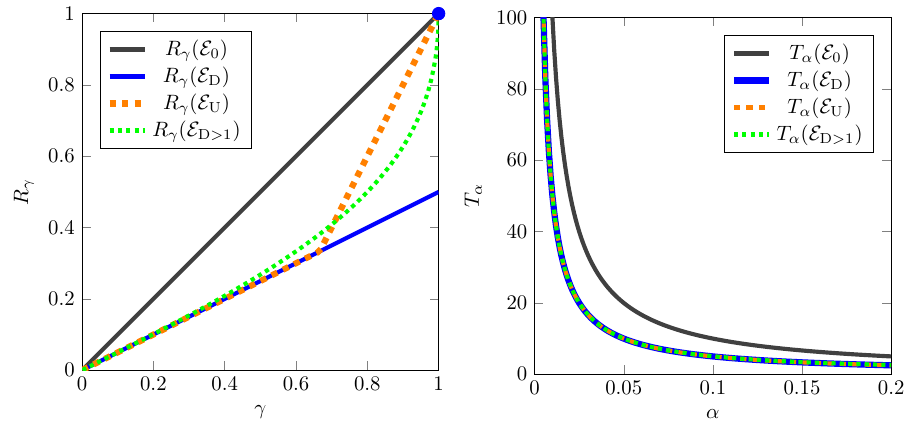}

\caption{Comparison of worst-case type-I errors and improved thresholds for decreasing and unimodal densities. }    \label{fig:threshold-comparison-1}
\end{figure}

\subsection{Distributions of log-transformed e-variables}

In many applications of e-values, the final e-variable $E$ is the product of many e-variables. In such a setting, assuming $\log E$ has some simple distributional properties may be reasonable. 
We say that a random variable $X$ has a symmetric distribution
if there exists $c\in \R$ such that $X$ and $c-X$ are identically distributed. 
We consider six different sets in this section:
\begin{align*}
\mathcal E_{\rm LS}& =\{E\in\mathcal E_0: \mbox{$\log E$ has a symmetric distribution}\},
 \\
 \mathcal E_{\rm LU}& =\{E\in\mathcal E_0: \mbox{$\log E$ has a unimodal density}\},
 \\ \mathcal E_{\rm LD>0}& =\{E\in\mathcal E_0: \mbox{$\log E$ has a decreasing density over $[0, \infty)$}\},
\\ \mathcal E_{\rm LD}& =\{E\in\mathcal E_0: \mbox{$\log E$ has a decreasing density}\},
\\
 \mathcal E_{\rm LUS}& =\{E\in\mathcal E_0: \mbox{$\log E$ has a unimodal and symmetric distribution}\},
 \\
\mathcal E_{\rm LN}& =\{E\in\mathcal E_0: \mbox{$E$  has a log-normal distribution}\}.
 \end{align*} 
 In all sets above, we require $\p(E=0)=0$, so that $\log E$ is a real-valued random variable. 
Note that $\mathcal E_{\rm LN}\subseteq \mathcal E_{\rm LUS} \subseteq \mathcal E_{\rm LS} $ and $\mathcal E_{\rm LD}\subseteq \mathcal E_{\rm LD>0}$.
The point-mass distributions at some $x\in (0,1]$ are included in all sets above, which are degenerate cases of log-normal distributions. 

The following theorem provides the worst-case type-I errors for the six sets of e-variables. The proof is given in Appendix \ref{app:proof-log-transformed}.

\begin{theorem} \label{th:log-transformed}
For $\gamma \in (0,1]$, 
\begin{enumerate}
\item [(i)]
     $R_\gamma (\mathcal E_{\rm LS}   ) =\gamma \wedge (1/2)$  if $\gamma\ne 1$ and $R_1(\mathcal E_{\rm LS}  ) =1$;
     \item[(ii)] $R_\gamma (\mathcal E_{\rm LU}   ) =\gamma$;
     
     \item[(iii)] $\mathcal{E}_{\rm LD>0} \subseteq \mathcal{E}_{\rm D>1}$ and 
$$ 
\frac{\gamma}{e - \gamma} \le 
R_\gamma(\mathcal{E}_{\rm LD>0}) = 
\frac{\gamma}{t_\gamma}  ,$$ 
where $t_\gamma$ is the unique solution of $t(1-\log t) = \gamma$ for  $t \in (1, e)$,
and we have 
$$
t\uparrow e \mbox{~~~and~~~}
\frac{ R_\gamma(\mathcal{E}_{\rm LD>0}) }
{\gamma } \downarrow \frac 1 e \mbox{~~~as $\gamma\downarrow 0$};
$$ 
     
     \item[(iv)] 
     $R_\gamma (\mathcal E_{\rm LD} ) =R_\gamma (\mathcal E_{\rm LUS}   ) $
 and     $$\frac{\gamma}e\le R_\gamma (\mathcal E_{\rm LD}   ) =R_\gamma (\mathcal E_{\rm LUS}   ) 
  \le \frac{\gamma}{e(1-\gamma^2)} \wedge \left(   
\frac{\gamma }{1+ \sqrt{1-\gamma^2}} \right);$$
     
     \item[(v)]
     if $\gamma\ne1$ then $$R_\gamma (\mathcal E_{\rm LN}   ) =\Phi\left(-\sqrt{-2\log\gamma}\right) \leq  \frac{\gamma}{2\sqrt{-\pi \log \gamma}},$$
     where $\Phi$ is the standard normal cdf, and $R_1(\mathcal E_{\rm LN}   )=1$.
     \end{enumerate}
\end{theorem}

The results in parts (i) and (ii) of Theorem \ref{th:log-transformed} show that the assumptions of log-transformed symmetric or log-transformed unimodal distributions do not improve the bound from Markov's inequality for the most useful cases where $\gamma \le 1/2$. 

The class $\mathcal E_{\rm LD>0}$ offers a non-trivial threshold improvement by roughly a factor of $e$. As with $\mathcal E _{\rm D>1}$, most evidence against the null come from observing e-values larger than $1$, which is equivalent to observing $\log$ e-values larger than zero. For $\gamma$ small, such as $0.01$ or $0.05$, the lower and upper bounds in part (iv) of Theorem \ref{th:log-transformed} are very close; hence, using the upper bound is not wasteful. Part (v) shows that log-normality improves over $R_\gamma(\mathcal E_0)$ by at least $3.54\sqrt{-\log \gamma}$. 

We summarize the bounds on worst-case type-I errors and improved thresholds from Theorem \ref{th:log-transformed} in Figure \ref{fig:threshold-comparison-3}.

\begin{figure}[ht]
    \centering

    \includegraphics[width = \textwidth]{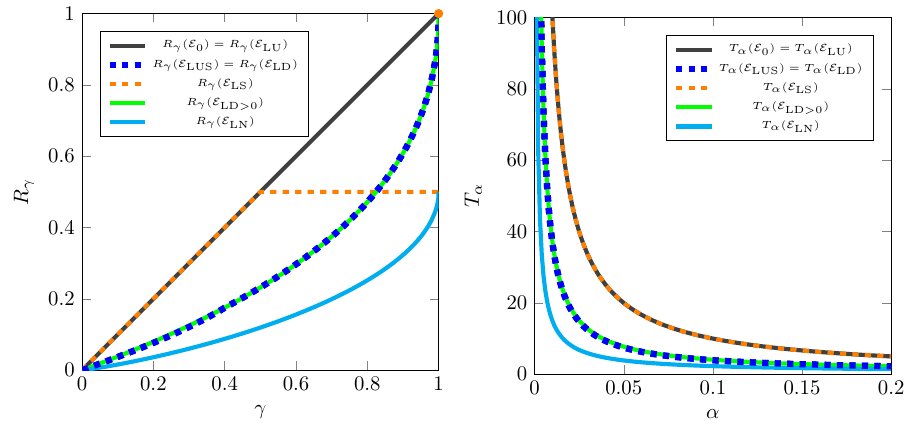}

    \caption{Comparison of worst-case type-I errors and improved thresholds for log-transformed e-variables. Results for $\mathcal{E}_{\rm LUS}$ are conservative bounds since we only find an upper bound for $R_\gamma(\mathcal{E}_{\rm LUS})$ and $T_\alpha(\mathcal{E}_{\rm LUS})$.}
    \label{fig:threshold-comparison-3}
\end{figure}

\subsection{Log-concave density, survival or distribution functions}

In this section, we study distributional assumptions of log-concavity, a popular property for density functions, survival functions, and distribution functions. 
We will consider the sets 
\begin{align*}
\mathcal E_{\rm LCD}&=\{E\in\mathcal E_0: \mbox{$E$  has a  log-concave   density    on $\R$}\},
 \\
\mathcal E_{\rm LCS}&=\{E\in\mathcal E_0: \mbox{$E$  has a log-concave    survival function on $\R$}\},
\\ 
\mathcal E_{\rm LCF}&=\{E\in\mathcal E_0: \mbox{$E$  has a log-concave    distribution function on $\R$}\}.
 \end{align*} 

The density of $E\in\mathcal E_{\rm LCD} $ may take the value of $0$ on some intervals like $(-\infty,a)$ or $(a,\infty)$ for $a\ge 0$, or both, that is, the log-density function is $-\infty$ on such intervals.
The assumption of log-concave density is satisfied by, for instance, normal, uniform, Laplace and gamma (with shape parameter $\ge 1$) distributions, as well as their location-scale transforms. When considering e-variables, the support of these distributions should be confined to $[0,\infty)$.

All distributions with log-concave densities have log-concave survival functions and log-concave distribution functions, thus $\mathcal E_{\rm LCD} \subseteq \mathcal E_{\rm LCS}$ and $\mathcal E_{\rm LCD} \subseteq \mathcal E_{\rm LCF}$.
Moreover, $\mathcal E_{\rm LCD} \subseteq \mathcal E_{\rm U}$ and $\mathcal E_{\rm D} \subseteq \mathcal E_{\rm LCF}$.  
The property of having a log-concave survival function is also known as having an increasing hazard rate or increasing failure rate. 
Log-normal and Pareto distributions belong to the class of distributions with log-concave cumulative distribution functions; however, they do not have log-concave densities or survival functions. 
For many examples of distributions belonging to the above classes, see, e.g., \cite{bagnoli2005logconcave}. In \cite{shah2013variable}, the authors derive probability bounds for r-concave distributions in the context of stability selection. R-concave distributions can be seen as an interpolation between log-concavity and unimodality, but the results of \cite{shah2013variable} are formulated on discrete grids and not compatible with the setting in our paper.  

\begin{theorem}\label{thm:log-distribution}
	For $\gamma\in (0,1],$ we have
	\begin{itemize}
		\item[(i)] $e^{-1/\gamma} \le R_{\gamma}(\mathcal E_{\rm LCS} )
= e^{s_\gamma/\gamma}
\le  e^{1-1/ \gamma  },$ where $s_1=0$ and for $\gamma<1$, $s_\gamma$ is the unique solution $s\in (-1,0)$ to the equation 
$ e^{s/\gamma} =s +1$;
		\item[(ii)] $\mathcal{E}_{\rm LCD} \subseteq \mathcal{E}_{\rm LCS}$ and $\mathcal{E}_{\rm LCD} \subseteq \mathcal{E}_{\rm U}$ and thus
		$$e^{-1/ \gamma  } \le R_\gamma(\mathcal{E}_{\rm LCD}) \leq R_\gamma(\mathcal{E}_{\rm U}) \wedge R_\gamma(\mathcal{E}_{\rm LCS})  \le  e^{1-1/ \gamma  };$$
		\item[(iii)] $R_{\gamma}(\mathcal E_{\rm LCF} )= \gamma$.
	\end{itemize}
\end{theorem}

The proof of Theorem \ref{thm:log-distribution} is given in Appendix \ref{app:proof-log-distribution}.

Part (i), dealing with log-concave survival functions, gives the most useful result since it is easy to verify visually and provides a useful bound. The upper bound in part (i) of Theorem \ref{thm:log-distribution} is convenient for applications due to its analytical form. Numerically, we find that $R_{\gamma}(\mathcal{E}_{\rm LCS})$ is closer to the lower bound of $e^{-1/\gamma}$, but using the upper bound is not so wasteful since it is at most the constant $e$ too large, while the bound grows exponentially. In part (ii), we could not obtain the distribution that attains the lower bound, implying that there may be room for improvement. For the class of log-concave distributions, we obtain a negative result: There is generally no better bound on $R_{\gamma}(\mathcal E_{\rm LCF} )$ other than the Markov bound.
 
We summarize the worst-case type-I error bounds and the improved thresholds in Figure \ref{fig:threshold-comparison-2}. Recall that the bounds for $\mathcal{E}_{\rm LCD}$ do not hold with equality, so there may be better bounds available, although we have not identified them. 
\begin{figure}[ht]
    \centering

    \includegraphics[width = \textwidth]{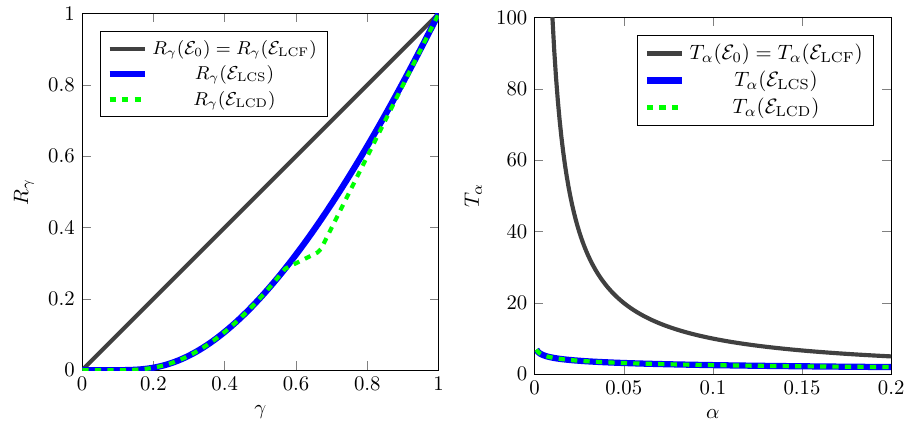}

\caption{Comparison of worst-case type-I errors and improved thresholds for log-concave density, survival and distribution functions. Results for $\mathcal{E}_{\rm LCD}$ are conservative bounds since we only find an upper bound for $R_\gamma(\mathcal{E}_{\rm LCD})$ and $T_\alpha(\mathcal{E}_{\rm LCD})$.}   \label{fig:threshold-comparison-2}
\end{figure}

\subsection{Summary}

For a class of e-variables $\mathcal{E}$, Lemma \ref{lem:1} allows one to design a more aggressive e-test based on the rejection threshold $T_{\alpha}(\mathcal{E})$. The rejection region is enlarged, leading to rejections with less evidence while maintaining the type-I error controlled at $\alpha$ for $\alpha \in (0, 1)$. We have already provided plots of the worst-case type-I errors and improved rejection thresholds in Figures \ref{fig:threshold-comparison-1},  \ref{fig:threshold-comparison-3}, and \ref{fig:threshold-comparison-2}. We summarize the results on the improved thresholds in Table \ref{tab:improved-threshold}.

\begin{table}[ht]
	\centering
	\caption{Improved thresholds $T_{\alpha}(\mathcal{E})$ for different distributional hypotheses.}\label{tab:improved-threshold}
	\begin{tabular}{rrrrrrr}
                                                           &         \multicolumn{6}{c}{$\alpha$}         \\
                                                           &  0.001 &  0.01 &  0.02 &  0.05 &  0.1 &  0.2 \\ \cmidrule(lr){2-7}
		       $\mathcal{E}_{0}, \mathcal{E}_{\rm LS}$  &   1000 &   100 &    50 &    20 &   10 &    5 \\
		    $\mathcal{E}_{\rm D}, \mathcal{E}_{\rm U}$ &    500 &    50 &    25 &    10 &    5 &  2.5 \\
		                       $\mathcal{E}_{\rm D>1}$  &    500 & 50.01 & 25.01 & 10.03 & 5.05 & 2.60 \\
		$\mathcal{E}_{\rm LCD}, \mathcal{E}_{\rm LCS}$ &   6.91 &  4.65 &     4 &  3.15 & 2.56 &    2 \\
		 $\mathcal{E}_{\rm LUS}, \mathcal{E}_{\rm LD}$    & 367.88 & 36.82 & 18.45 &  7.49 & 3.93 & 2.25 \\
		                      $\mathcal{E}_{\rm LD>0}$   & 368.25 & 37.16 & 18.77 &  7.73 & 4.07 & 2.25 \\
		                        $\mathcal{E}_{\rm LN}$ &    118 & 14.97 &  8.24 &  3.87 & 2.27 & 1.42
	\end{tabular}
\end{table}

 There are some ad-hoc thresholds of e-values used in the literature, as it is clear that the standard calibrator $\min(1/p, 1)$ is too restrictive.
For instance,  
\cite{S21} recommended using the conversion formula $1/\sqrt{p} - 1$. 
\cite{VW21} suggested using the rule of thumb proposed in 
Appendix B of \cite{jeffreys1961theory} to convert between p-values and Bayes factors,
as Bayes factors align with e-values in the context of simple hypotheses. Both Shafer's conversion and Jeffreys' rule of thumb are slightly less conservative than the class $\mathcal{E}_{\rm LN}$ but more conservative than $\mathcal{E}_{\rm LCS}$. 
In contrast to these ad-hoc choices, the thresholds obtained in this paper have a provable type-I error guarantee under the respective assumptions.


\section{Perturbation bounds for improved thresholds}\label{sec:perturbation}

Suppose that we have a class of e-variables $\mathcal{E}$ along with its worst-case probability bound $R_{\gamma}(\mathcal{E})$. We will study the worst-case probability bound for a perturbed version of $\mathcal{E}$ and study how these bounds depend on the chosen perturbation. We equip the space of probability distributions with the L\'evy-Prokhorov metric, which metrizes weak convergence: for two probability measures $\mu$ and $\mu'$ on $\R$, we define
$$d_{\mathrm{LP}}(\mu, \mu') = \inf \left\{\delta > 0:
\begin{array}{l}
    \mu(A) \leq \mu'(A^\delta) + \delta \\
    \mu'(A) \leq \mu(A^\delta) + \delta 
\end{array}
~\text{  for all Borel } A\subseteq \R\right\},$$
where 
$$A^\delta = \{x \in \mathbb{R}: |x-y| < \delta \mbox{ for some $y\in A$}\} \mbox{ for $\delta>0$}.$$
Define the $\delta$-neighbourhood of $\mathcal{E}$  for $\delta>0$ as 
$$\mathcal{E}^\delta = \{E' :  d_{\mathrm{LP}}(\mu_E, \mu_{E'}) < \delta \mbox{ for some $E\in \mathcal E$}\},$$
where $\mu_E$  is  the distribution of $E$ for any random variable $E$.
Denote by 
$$R_\gamma(\mathcal{E}^\delta) = \sup_{E' \in \mathcal{E}^{\delta}} \p(E' \geq 1/\gamma),$$
the worst-case probability bound for random variables in $\mathcal{E}^\delta$. Note that $\mathcal{E}^\delta$ need not only contain e-variables, hence the notion of worst-case probability bound is more general than in the previous sections. 

Our motivation in studying $R_{\gamma}(\mathcal{E}^\delta)$ is robustness and transfer via weak convergence. By replacing $\mathcal{E}$ with its $\delta$-neighbourhood, we get thresholds that tolerate small model misspecifications (for instance, if an approximate e-variable $E'$ is approximately unimodal, we can still use $R_\gamma(\mathcal{E}_U)$ to derive thresholds).
Worst-case probability bounds for a sequence $(E_n)_{n\in \N}$ carry over to its limit $E$ in distribution, as for asymptotic e-values (see \cite{IWR25}). 

\begin{proposition}
    For any $\gamma \in (0, 1]$ and $0 < \delta < 1/\gamma$, we have
    \begin{equation} 
    \label{eq:R1-p2} R_\gamma(\mathcal{E}^\delta) \leq R_{\frac{\gamma}{1 - \gamma \delta}}(\mathcal{E}) + \delta.
    \end{equation}
\end{proposition}

\begin{proof}
Fix some $E' \in \mathcal{E}^\delta$. By the definition of $\mathcal{E}^\delta$, there exists   $E \in \mathcal{E}$ such that $d_{\mathrm{LP}}(\mu_E, \mu_{E'}) < \delta$. Consider the closed set $A = [1/\gamma, \infty)$. Then, $A^\delta = [1/\gamma - \delta, \infty)$. From the definition of $d_{\mathrm{LP}}$, we have
$$\p(E' > 1/\gamma) = \mu_{E'}(A) \leq \mu_{E}(A^\delta) + \delta = \p(E > 1/\gamma - \delta) + \delta.
$$
Taking the supremum over all $E \in \mathcal{E}$ and then over all $E' \in \mathcal{E}^\delta$, we obtain \eqref{eq:R1-p2}.
\end{proof}

\section{Supremum of comonotonic e-values}\label{sec:como}

In this section, we will introduce a testing strategy for situations where we have access to a family of comonotonic e-variables. This strategy can be used in conjunction with the improved thresholds. 

\subsection{Threshold for  comonotonic e-values}

In what follows, each $Q_\theta$ is a probability measure. To test $\{Q_{\theta_0}\}$ against $\{Q_{\theta}\}$, we usually use the e-variable as a likelihood ratio, that is, 
$$
E_\theta = \frac{\d Q_\theta}{\d Q_{\theta_0}}.
$$
If we test   $\{Q_{\theta_0}\}$ against $\{Q_\theta:\theta \in \Theta_1 \}$ where $\Theta_1\subseteq \Theta$, we can use a mixture 
$$
E_\nu = \int_{\Theta_1} \frac{\d Q_\theta}{ \d Q_{\theta_0}}  \nu(\d \theta),
$$
where $\nu$ is a prior probability on $\Theta_1$. 

For instance, as in \cite{VW21}, the simplest example is to test $\mathrm{N}(0,1)$ against $\mathrm{N}(\mu,1)$ with $\mu\ne 0$ for iid observations $X_1,\dots,X_n$. 
The e-variable $E_\mu$  for $\mathrm{N}(0,1)$ is   the likelihood ratio 
$$
E_\mu =\exp( {\mu S_n -n\mu^2/2}),
$$ 
where $S_n=\sum_{i=1}^n X_i$. 
Suppose that the alternative hypothesis is $\mu>0$. In this case, an interesting observation is that $E_\mu$, $\mu>0$ are \emph{comonotonic} random variables.  
Recall that random variables $E_{\theta}$, $\theta\in \Theta$
are comonotonic if there exists a common random variable $Z$ such that $E_\theta$ is an increasing function of $Z$ for each $\theta\in \Theta$. 

\begin{lemma}\label{lemma:comonotonic-survival}
	Let $x\in \R$ and $\gamma \in [0,1]$. 
	Suppose that $\X$ is a collection of comonotonic random variables
	satisfying $\p(X \ge  x)\le \gamma $ for each $X\in \X$.
	Then $\p(\sup_{X\in \X} X \ge x )\le \gamma$.
\end{lemma}
\begin{proof}
	Let $q_X$ be the right quantile function of $X$ for each $X$, that is,
	$$
	q_X (\alpha) = \inf\{ x\in \R: \p(X \le x)  >  \alpha\},~~~\alpha \in (0,1).
	$$
	Note that $\p(X \ge x)\le \gamma $ is equivalent to 
	$q_X (1-\gamma)\le x$ (e.g., Lemma 1 of \cite{GJW23}).
	This gives  
	$\sup_{X \in \X} q_X (1-\gamma)\le x$.
	Note that comonotonicity implies that 
	$\sup_{X \in \X} q_X (1-\gamma)$ is the right quantile function of $\sup_{X\in \X} X$.
	Using the above relation again we get  $\p(\sup_{X\in \X} X >x )\le \gamma$.
\end{proof}

The next result follows directly.
\begin{proposition} \label{prop:sup-comonotonic}
	Suppose that $
\mathcal E=\{E_{\theta}: \theta\in \Theta \}\subseteq \mathcal E_0$ is a collection of comonotonic e-variables for  a hypothesis $  H$. 
	Then $\sup_{Q\in  H} Q(\sup_{\theta\in \Theta} E_\theta \ge T_\alpha(\mathcal E))\le \alpha$.
\end{proposition} 

Consequently, we can obtain a valid test using the supremum of e-variables instead of their mixture under comonotonicity.
This will always have a higher power than using a mixture. 

If we push this idea further, we can take the set $\mathcal E$ of all e-variables for $\p$ that are decreasing functions of a statistic. This will result in an exact p-variable. 
\begin{proposition}\label{prop:function-precise}
	Let $\mathcal E$ be the set of all e-variables for $\p$ that are decreasing functions of a random variable $X$ with distribution function $F$. 
	Then, $(\sup_{E\in \mathcal E}E )^{-1}= F(X)$ almost surely, i.e., the natural p-variable built from $X$.  
\end{proposition}
\begin{proof}
	For $s \in \R$, let $$E_s = \frac1 {F(s)}\id_{\{X\le s\}} \mbox{~ with the convention $0/0=1$},$$
	which is an e-variable, decreasing in $X$.
	It is clear that 
	$\sup_{s\in \R}E_s =1/F(X) $,
	and  hence 
	$\sup_{E\in \mathcal E}E \ge 1/ F(X)$.
	To show the opposite, 
	suppose that there exists $E\in \mathcal E$ such that $$\p\left(E>\frac1{F(X)}\right)>0.$$ 
	Since $E$ is decreasing in $X$,
	there exists $s\in \R$ such that $\p(E>E_s)>0$ and $\p(E\ge E_s)=1$. 
	This contradicts $\E[E]\le1$.
\end{proof}

Proposition \ref{prop:function-precise} suggests that the supremum of comonotonic e-values can be seen as a conceptual middle ground between an e-value ($\mathcal E$ is a singleton)  and a p-value ($\mathcal E$ in Proposition \ref{prop:function-precise}, which is a maximal comonotonic set).

The idea of taking a supremum does not generalize to optional stopping for the e-process because comonotonicity is not preserved in general. That is, for e-processes $n\mapsto E_\theta(n)$ indexed by $\theta$, even if the e-variables $E_\theta(n)$ are comonotonic across $\theta$ for each fixed $n$ (as in the normal example), the stopped e-variables $E_\theta(\tau)$ need not be comonotonic across $\theta$ for a stopping time $\tau$. 

\subsection{Implications in  the Gaussian setting}

When testing $\mathrm{N}(0,1)$ against $\{\mathrm{N}(\mu,1): \mu>0\}$, the likelihood ratio e-variables are comonotonic. Instead of using 
$$
E_\mu(n) =\exp( {\mu S_n - n \mu^2/2})
$$ 
for a fixed $\mu>0$ or its mixture, we can use
$$
Y(n)=\sup_{\mu>0} \exp( {\mu S_n  - n \mu^2/2}) ,
$$ 
which, although not being an e-variable, gives $\p(Y(n)\ge 1/\alpha)\le \alpha$ under the null. 
We can explicitly compute
$$
Y(n)=  \exp\left(\frac{(S_n)_+^2}{2n}\right).
$$ 
We can compare this with (23) of \cite{VW21}, which gives the e-variable $$ E_\Phi(n)= \frac{1}{\sqrt{n+1}} \exp\left(\frac{S_n^2}{2n+2}\right),$$
obtained via a mixture of $E_\mu(n)$ using a  standard normal prior $\Phi$ (a two-sided alternative).
Under the null, we can compute 
$$
\p(Y(n) \ge 1/\alpha) = \p \left(\frac{(S_n)_+^2}{2n} \ge - \log \alpha\right)=1-\Phi \left(  \sqrt{-2 \log \alpha}\right),
$$
where $\Phi$ is the standard normal cdf. Denote by $\phi$ the standard normal pdf. Since for $x>0$, 
$$
\frac{x}{x^2+1} \le  \frac{1-\Phi(x)}{\phi(x)}\le \frac 1x ,
$$
we get
$$ \p(Y(n) \ge 1/\alpha) \approx \frac{1}{\sqrt{-2\log \alpha}} \frac{1}{\sqrt {2\pi}} \alpha.$$
Although this is still less than $\alpha$, the magnitude has improved substantially from $E_\Phi(n)$, which roughly corresponds to replacing $\alpha $ by $\alpha^{(n+1)^{-1/2}}$.


\begin{remark}
  We also note that $E_\mu$, $\mu>0$ are   log-normally distributed, so we can use    the threshold   $T_\alpha(\mathcal{E}_{\rm LN})  = \exp(-(\Phi^{-1}(\alpha))^2/2),$
where $\Phi^{-1}$ is the inverse distribution function (quantile function) of a standard normal distribution. Then, it is easy to show  $\p\left(Y(n) > T_\alpha(\mathcal{E}_{\rm LN})\right) = \alpha$ for $\alpha \in (0, 1)$; that is, we get an exact type-I error. 
\end{remark}

\subsection{Maximum likelihood e-values}
\label{sec:43}

The test constructed in the previous section is based on taking the supremum over a collection of Gaussian likelihood ratio e-variables. We now investigate how to generalize this to other likelihood ratio e-variables. 

Suppose that we test $Q_{\theta_0}$ against $\{Q_{\theta} : \theta \in \Theta\}$ with a likelihood ratio e-variable,
$$E_\theta(n) := \prod_{i = 1}^{n} \frac{\d Q_{\theta}}{\d Q_{\theta_{0}}}(x_{i}).$$
If $\{E_\theta(n) : \theta \in \Theta\}$ is a collection of comonotonic random variables, we can take the supremum according to Proposition \ref{prop:sup-comonotonic}. We have 
$$\underset{\theta \in \Theta}{\arg\max} \prod_{i = 1}^{n} \frac{\d Q_{\theta}}{\d Q_{\theta_0}}(x_{i}) = \underset{\theta \in \Theta}{\arg\max} \prod_{i = 1}^{n} \left.\frac{\d Q_{\theta}(x)}{\d x}\right\vert_{x = x_i} ,$$
corresponding to the maximum likelihood estimator (MLE) over $\Theta$, denoted $\hat\theta$. Therefore, 
$$\underset{\theta \in \Theta}{\sup} E_\theta(n) = E_{\hat \theta}(n),$$
and we can use an e-value constructed with a plug-in estimator using the entire dataset instead of splitting the data as in the split likelihood-ratio e-test of \cite{WRB20}.

For $\{E_\theta(n) : \theta \in \Theta\}$ to be a collection of comonotonic random variables, we require that, for every $\theta \in \Theta$, $({\d Q_{\theta}}/{\d Q_{\theta_{0}}})(x)$ increases in $x$ over the union of the supports of $Q_{\theta}$ and $Q_{\theta_0}$. The latter condition coincides with the definition of the likelihood ratio order (see, for instance, Section 1.C of \cite{SS07} for properties).

One simple way to verify if a class of likelihood ratio e-variables is comonotonic is within the exponential family. For a parameter ${\theta}$ (possibly a vector), a family of distributions is said to belong to the exponential family if its density function or probability mass function can be written as 
$$\frac{\d Q_{{\theta}}(x)}{\d x} = h(x) \exp\left\{{\eta}({\theta}) {T}(x) - {A}({\theta})\right\},$$
where ${\eta}({\theta})$ is the natural parameter, ${T}(x)$ is the summary statistic, and $A(\theta)$ is called the log-partition. When testing two distributions in the exponential family, the likelihood ratio e-value (for one observation) is 
$$\frac{\d Q_{\theta}}{\d Q_{\theta_{0}}}(x) = \exp\left\{({\eta}({\theta}) - {\eta}({\theta}_0)) {T}(x) - {A}({\theta}) + {A}({\theta}_0)\right\}.$$
The latter forms a comonotonic family if ${\eta}({\theta}) - {\eta}({\theta}_0)$ has the same sign for each $\theta$. We will give examples of maximum likelihood e-values within the gamma family in Section \ref{ss:gamma-como}.

\subsection{Comonotonic adjustment to Wilcoxon's signed-rank e-test}

Our previous examples of comonotonic e-values are based on likelihood ratios. Comonotonic e-variables can also arise in nonparametric settings; we now provide an example for testing symmetry. Wilcoxon's signed-rank test was developed as a nonparametric alternative to the standard paired $t$-test for symmetry. It tests the null hypothesis that the data-generating distribution (e.g., the distribution of the difference between a pair of variables) is symmetric about zero, versus an alternative of asymmetry.  \cite{vovk2024nonparametric} derived a signed-rank e-variable based on Wilcoxon's signed-rank test, given by 
\begin{equation}\label{eq:wilcoxon}
    E_{\beta} = \exp(\beta V_n) \prod_{i = 1}^n \frac{2}{1 + \exp(\beta i)},
\end{equation}
for a parameter $\beta \in \mathbb{R}$, where $V_n $ is the sum of the ranks of the positive observations in the data.  
The collection of e-variables in \eqref{eq:wilcoxon} is comonotonic for $\beta < 0$ or $\beta > 0$.
For example, we can use the threshold $1/\alpha$ with the statistic
$\sup_{\beta > 0} E_{\beta}$
and get a valid level-$\alpha$ test. Taking the supremum also removes the reliance on $\beta$ in the family of signed-rank e-tests. 


\section{Thresholds for stopped e-processes}\label{sec:combined}

Researchers using e-values frequently rely on independent or sequential data when collecting evidence for scientific discovery. To perform an e-test from these data, a common strategy is to compute an e-value for each observation and combine them to form an e-value or an e-process. Denote by $[T]=\{1,\dots,T\}$, where $T$ can be either finite or infinite (when $T=\infty$, we set $[T]=\N$). We will assume that e-variables $E_1,\dots, E_T$ are available. We will assume that the e-variables are either independent or sequential. E-variables $(E_t)_{t\in [T]}$ are said to be \emph{sequential} if $\E[E_t \vert \mathcal F_{t-1}] \leq 1$ for $t \in [T]$, where $(\mathcal F_t)_{t\in \{0,1,\dots,T\}}$ is a given filtration to which $(E_t)_{t\in [T]}$ is adapted.

If one has access to sequential e-values, one can construct an e-process $\{M_t : t \in [T]\}$, accumulating evidence from individual sequential e-variables. E-processes have the property that, for any stopping time $\tau$, the stopped e-process is an e-value. The property $\E[M_\tau] \le 1$ for any stopping time is called the anytime validity property of e-processes, implying that a statistician can stop an experiment early or continue an experiment after peeking at the results from an e-test while retaining type-I error control. To test with an e-process while controlling the type-I error at level $\alpha$, it is common to use Ville's inequality, stating for $\alpha \in (0, 1]$ that 
$$\p\left(\sup_{t \in [T]} M_t \ge \frac{1}{\alpha}\right) \le \alpha.$$
Therefore, the \textit{hit-and-stop} strategy lets the statistician observe the e-process and stop the process once it crosses $1/\alpha$, all while controlling the type-I error at level $\alpha$.

In general, tightening Ville's inequality is challenging. The inequality fails to be an equality since the e-process often overshoots the boundary when it first exceeds $1/\alpha$; a tighter bound could be obtained by accounting for the distribution of the excess over the boundary. For this reason, it is generally difficult to improve upon the $1/\alpha$ threshold in the {hit-and-stop} strategy. By considering other stopping times, we can use the improved thresholds. In the following result, we provide an instance that relies on a restrictive assumption on the stopping time. More general stopping times are left for future research.

A process $Y$ is called predictable if $Y_t$ is measurable with respect to $\mathcal{F}_{t-1}$ for each $t \in [T]$. A stopping time $\tau  $ is called predictable if $\{\tau = t\} \in \mathcal{F}_{t-1}$ for every $t \in [T]$ and it takes values in $[T]$. In the context of the current paper, a predictable e-process is one where the final e-variable in the stopped e-process is independent of the stopping time. We note that stopped e-processes using the \textit{hit-and-stop} strategy are not predictable.

\begin{theorem}
\label{th:r1-predict}
    Let $(E_t)_{t\in [T]}$ be a sequence of independent e-variables and $M_t = \prod_{s = 1}^t E_s$. Let $\tau$ be a predictable  stopping time and assume that $E_t \in \mathcal{E}$ for $t\in [T]$.
    \begin{enumerate}
    \item [(i)] 
Let $r(x)= R_{x}(\mathcal E)$ for $x> 0$.
Then $r(x)\le x$ for $x>0$ and
 $$\p(M_\tau \geq 1/\gamma) \leq  \E[r(\gamma M_{\tau-1})] \le \gamma.$$    
    \item [(ii)] 
    
    Suppose there exist constants $\zeta, \theta\in (0,1]$ such that 
    ${R}_{\gamma}(\mathcal{E}) \leq \zeta \gamma$ for $\gamma \in [0, \theta)$ and let $\delta = \p(  \theta \le \gamma M_{\tau-1} < 1/\zeta)$. 
    Then, 
    $$\p(M_\tau \geq 1/\gamma) \le \zeta \gamma +  (1/\zeta-1)\delta  .$$
\item [(iii)]    If $\mathcal E$ is the class of e-variables with decreasing densities on $(0,\infty)$, then     $$\p(M_\tau \geq 1/\gamma) \leq \frac{ \gamma}2.$$\end{enumerate}
    
\end{theorem}

\begin{proof} 
We first show part (i).  
The fact that $r(x)\le x$ for $x>0$ is due to Markov's inequality.  
Let $S_t(y)=\p(E_t\ge y)$ for $y\in \R$ and $t\in [T]$, and by definition, it  satisfies $S_t(1/x)\le r(x)$ for $x>0$. 
Since $\{\tau=t\}$ is in $\mathcal F_{\tau-1}$, for any $Y$ that is measurable to $\mathcal F_{\tau-1}$, we have 
\begin{align*}
  \p(E_\tau \ge  Y\mid \mathcal F_{\tau-1})& =  
\sum_{t\in [T]} \E \left[\id_{\{E_t \ge  Y\}}  \mid \mathcal F_{\tau-1}  \right] \id_{\{\tau=t\}} 
\\& = \sum_{t\in [T]} S_t(Y)  \id_{\{\tau=t\}} \le 
  \sum_{t\in [T]}  r(1/Y)   \id_{\{\tau=t\}} =
r(1/Y).
\end{align*}   
Note that $M_{\tau} = M_{\tau - 1}E_\tau$ and $M_{\tau - 1}$ is $\mathcal{F}_{\tau - 1}$-measurable. 
Therefore, we get
 $$
    \p(E_\tau  \ge  1/(\gamma M_{\tau-1}) \ \vert \ \mathcal F_{\tau-1}) \leq r(\gamma M_{\tau-1}), 
 $$
and hence 
$$\p(M_{\tau} \ge 1/\gamma) =   
\E\left[\p\left(M_{\tau} > 1/\gamma \ \vert \ \mathcal{F}_{\tau-1}\right)\right] 
\le \E\left[r(   \gamma M_{\tau - 1})\right]  \le \E\left[   \gamma M_{\tau - 1} \right]  \le \gamma.
$$

To show part (ii),  let $A=\{ \theta \le \gamma M_{\tau-1} <1/\zeta\}$. We have 
\begin{align*}
\E[r(\gamma M_{\tau-1}) ]
  &\le \E\left[\zeta \gamma M_{\tau-1}\id_{A^c } + \gamma M_{\tau-1}\id_{A }\right]
  \\ &= \E\left[ \zeta \gamma  M_{\tau-1} +
   \gamma M_{\tau-1}(1-\zeta) \id_{A }  
  \right]
  \\& \le \zeta \gamma +  (1/\zeta-1)\delta.
\end{align*}

To show part (iii),  when $\mathcal{E}$ is the class of e-variables with decreasing densities on $(0, \infty)$, we have from Proposition 6.4 of \cite{W24} that $R_\gamma(\mathcal{E}) \le \gamma/2$ for every $\gamma>0$ (see also Remark \ref{rem:dec}), and hence the conclusion follows from part (i).
\end{proof}

If the stopping time $\tau$ is known in advance, as in classical p-tests,  then the stopping time is predictable. 
If the stopping time $\tau$ is independent of the e-process, we could include it in $\mathcal F_0$, so that it becomes predictable. 
Another example of a predictable stopping time is to stop one step (or more) after the e-process hits a threshold; one could imagine a termination notice on an experiment with a grace period.

The function $r$ in part (i)  of Theorem \ref{th:r1-predict}  
is computed in many results in Section \ref{sec:3}, and when explicit formula for $R_\gamma(\mathcal E)$ is not available, we can safely take $r$ as an upper bound on it.
The difficulty in using part (i) in practice is that we need the distribution of $M_{\tau-1}$ to compute $\E[r(\gamma M_{\tau-1})]$, which is often not available, although we know it is bounded above by $\gamma$.
The condition $\mathcal{R}_{\gamma}(\mathcal{E}) \leq \zeta \gamma$ for   $\gamma \in [0, \theta)$  in part (ii) 
is satisfied by many choices of $\mathcal E$ and moderate values of $\theta$ and $\zeta$ like $\theta=\zeta=0.6$; this can be easily verified in Figures \ref{fig:threshold-comparison-1}--\ref{fig:threshold-comparison-2}. The improvement is most useful when $\delta$ is small, which still requires some partial information about the distribution of $M_{\tau-1}$. 
On the other hand, the conclusion in part (iii) does not require any knowledge on the distribution of $M_{\tau-1}$.


\begin{remark}
In the special case that the stopping time $\tau$ is independent of the e-process $M$ satisfying $M_t \in \mathcal{E}$ for all $t\in [T]$,   it is straightforward to show that  $M_\tau$ is in the convex hull of $\mathcal{E}$ with respect to distribution mixtures, and we can use the improved threshold for $\mathcal E$ with this stopped e-process (see Proposition \ref{prop:R1-1}).
\end{remark}

\begin{remark}
For the conclusion in Theorem \ref{th:r1-predict}, we only need the distribution of $E_{\tau}$ given $\mathcal F_{\tau-1}$ satisfies the constraint in 
$ \mathcal E$.
For instance, if we know $\tau$ takes values in an interval $[t_1,t_2]$, then only $E_t$ for $t\in [t_1,t_2]$ needs to be in $\mathcal E$. 
\end{remark}

In Appendix \ref{app:combined}, we derive additional results on the improved thresholds for combined e-variables in a fixed-sample setting.

\section{Boosting e-values in the e-BH procedure}\label{sec:boosting}

The e-BH procedure of \cite{WR22}
is introduced to test $K$ hypotheses with a false discovery rate (\cite{BH95}) control. 
As with e-tests for hypothesis testing, 
the base e-BH procedure typically controls the FDR 
at a more conservative level than the desired (or target) FDR. 
To improve the power of the e-BH procedure, we can boost the e-values
in the e-BH procedure such that it approaches the desired FDR level, 
while still controlling it.  
There are two known ways to accomplish this task. 
The first is to use distributional information 
on the marginal e-variables to boost the e-values
individually. This was proposed by \cite{WR22} already. 
Another method to boost the e-values is to use a relaxed definition of 
e-values, as in \cite{RB24}, possibly supplementing this with 
dependence information through conditional calibration as in \cite{LR24}. 

Since this paper aims to incorporate distributional assumptions through the marginal information on e-variables, we will focus on marginal boosting. We admit that the results in this section are mainly exploratory. The only positive results we obtained require strong assumptions (log-concave survival functions) and may not be readily applicable in practical situations. 
The proofs for the results in this section are given in Appendix \ref{app:boosting-proof}.

To boost e-values in the e-BH procedure through marginal boosting, we need to find constants $b_k$ 
in \citet[Equations (8) and (9)]{WR22}. 
That is, for an e-variable $E$, to find 
a constant $b>0$ (the larger, the better) such that 
    \begin{align}
 \E[T(\alpha b E )]\le \alpha,
    \label{eq:enhance-2}  
 \end{align} 
where $\mathcal K := [K]$, 
$K/\mathcal K :=\{K/k:k\in \mathcal K\}$, and $T:[0,\infty]\to [0,K]$ is defined by 
 $$ 
T(x)  = \frac{K}{\lceil  K/x\rceil}\id_{\{x\ge 1\}}\mbox{~~with $T(\infty)=K$}.
 $$
Noting that $\E[ T(\alpha E)]\le \E[\alpha E]\le  \alpha$, we know that $b = 1$ satisfies \eqref{eq:enhance-2}.  
If $b\mapsto  \E[T(\alpha b E )]$ is continuous (it suffices if $E$ has density on $(0,\infty)$), then we will take 
$$
b= \max\{ c \ge 1: \E[T(\alpha c E )]\le \alpha\}$$
under arbitrary dependence (AD). If the e-values in $K$ hypotheses satisfy the property of positive  regression dependence on a subset (PRDS), then $b$ can be improved to a constant satisfying 
 \begin{align}
 \label{eq:enhance-1}    \max_{x\in K/\mathcal K}  x  {\p(\alpha b E  \ge x)}   
 \le \alpha  .
  \end{align}
  Since $T$ depends on $K$,  it may not be convenient to analyze. We can relax \eqref{eq:enhance-2}, by using $T(x) \le x\id_{\{x\ge 1\}} $ for $x\in \R$, to 
      \begin{align}
 \E \left[ \alpha b E \id_{\{ \alpha b E \ge 1\}}\right]\le \alpha,
    \label{eq:enhance-3}  
 \end{align}
and  relax \eqref{eq:enhance-1} to 
       \begin{align}
\max_{x \ge 1}  x  {\p(\alpha b E  \ge x)}   \le \alpha.
    \label{eq:enhance-4}  
 \end{align}
 We define the quantity $B^{\rm AD}_\alpha (\mathcal E)$ 
 as 
 $$B^{\rm AD}_\alpha (\mathcal E) = \inf_{E\in \mathcal E}\sup\left\{ c \ge 1:  \E \left[ \alpha c E \id_{\{ \alpha c E \ge 1\}}\right]\le \alpha \right\},$$
 which is the boosting factor under AD for $E\in \mathcal E$ assuming continuity.
Similarly, 
we define  $$B^{\rm PR}_\alpha (\mathcal E) = \inf_{E\in \mathcal E}\sup\left\{ c \ge 1:\max_{x\ge 1}  x  {\p(\alpha c E  \ge x)}   \le \alpha \right\},$$
which is the boosting factor for $E\in \mathcal E$, assuming continuity in a PRDS multiple testing problem.

We first present a negative result on decreasing densities. This result suggests that getting a useful boosting factor for other similar distributional information may be difficult. In particular, there is no hope for any set containing $\mathcal E_{\rm D}$, such as $\mathcal E_{\rm U}$ and $\mathcal E_{\rm D>1}$.
\begin{proposition}\label{prop:e-bh-decreasing}
	For any $\alpha \in (0, 1)$, we have $B^{\rm AD}_\alpha (\mathcal E_{\rm D}) = 1$.
\end{proposition}

Next, two positive results concern e-variables with log-concave survival functions.  

\begin{proposition}\label{th:ad-ld}
	For some $\alpha \in (0, 1)$, let $c^{\rm AD}_1(\alpha) $ be  the constant $b\ge 1$ such that $\exp\left\{- 1 /(\alpha b)\right\}(1 + \alpha b) = \alpha \exp\{-1\}$ and $c^{\rm AD}_2(\alpha)$ be the constant $b'\ge 1$ such that $\exp\left\{- 1 /(\alpha b')\right\}(1 + \alpha b') = \alpha$. Then, we have
$$c^{\rm AD}_1(\alpha)\le B^{\rm AD}_\alpha\left(\mathcal{E}_{\rm LCS}\right) \le c^{\rm AD}_2(\alpha).$$
\end{proposition}

In Appendix \ref{app:boosting-proof}, we show that $\lim_{\alpha \to 0^+} \frac{c_2^{\rm AD}(\alpha)}{c_1^{\rm AD}(\alpha)} = 1$, implying that the lower and upper bounds in Proposition \ref{th:ad-ld} coincide when $\alpha$ is approaches zero. The next result gives bounds for boosting factors of e-variables with log-concave densities under PRDS. 
\begin{proposition}\label{th:prds-ld}
Let $\alpha \in (0, 1)$. Define 
 $$c_1^{\rm PR}(\alpha) = 
     \frac{1}{\alpha - \alpha \log \alpha} 
 ~~~~\mbox{and}~~~~  c_2^{\rm PR}(\alpha) = \begin{cases}
	    \exp\{1\}, & \alpha \ge \exp\{-1\};\\
     -\frac{1}{\alpha \log \alpha}, & \alpha \le \exp\{-1\}.
	\end{cases}$$ 
        We have  
		$$c_1^{\rm PR}(\alpha) \le B_{\alpha}^{\mathrm{PR}}\left(\mathcal E_{\rm LCS}\right) \le c_2^{\rm PR}(\alpha).$$
\end{proposition}

We can obtain numerically the bounds in Propositions \ref{th:ad-ld} and \ref{th:prds-ld}. We present the lower and upper bounds for $B^{\rm AD}_{\alpha}\left(\mathcal E_{\rm LCS}\right)$ and $B_{\alpha}^{\mathrm{PR}}\left(\mathcal E_{\rm LCS}\right)$ in Table \ref{tab:ld-boosting}. We note that $c_1^{\rm AD}(\alpha)$ is only slightly smaller than $c_1^{\rm PR}(\alpha)$, so one doesn't gain much boosting power by assuming PRDS (and the p-BH procedure is valid under PRDS and typically makes more discoveries, so the e-BH procedure may not be the best procedure to use in this case).

\begin{table}[ht]
	\centering
	\caption{Lower and upper bounds for $B^{\rm AD}_{\alpha}\left(\mathcal E_{\rm LCS}\right)$ and $B_{\alpha}^{\mathrm{PR}}\left(\mathcal E_{\rm LCS}\right)$.}\label{tab:ld-boosting}
	\begin{tabular}{rrrrr}
		$\alpha$ & $c_1^{\rm AD}(\alpha)$ & $c_2^{\rm AD}(\alpha)$ & $c_1^{\rm PR}(\alpha)$ & $c_2^{\rm PR}(\alpha)$ \\ \hline
		    0.01 &                  17.35 &                  20.86 &                  17.84 &                  21.71 \\
		    0.02 &                   9.82 &                  12.11 &                  10.18 &                  12.78 \\
		    0.05 &                   4.75 &                   6.13 &                   5.01 &                   6.68 \\
		    0.10 &                   2.82 &                   3.81 &                   3.03 &                   4.34
	\end{tabular}
\end{table}

\section{Simulation studies}
\label{sec:simulation}

In this section, we utilize the results derived in this paper to demonstrate how the improved thresholds enhance the power of e-tests through simulation. 

\subsection{Tests within the Gaussian family}\label{ss:simulation-gaussian}

A simple example illustrates the difference between e-processes and the supremum process. Suppose we have a sequence of $N$ iid observations $X_1, \dots, X_N$ from $N(\delta, 1)$ with $\delta = 0.3$ and let $S_n = X_1 + \dots + X_n$ for $n\in [N]$. 
We want to test the null hypothesis that data are $\mathrm{N}(0, 1)$ distributed
using different test statistics: the likelihood ratio e-process $\{E_\mu(n)\}_{n\in [N]}$ with the true alternative $\mu = \delta$ and two likelihood ratio e-processes with misspecified alternatives $\mu = 0.2$ and $\mu = 0.4$. We further compare with the universal e-process $\{E_\Phi(n)\}_{n\in [N]}$ and the supremum process $\{Y(n)\}_{n\in [N]}$. Note that the supremum process controls the type-I error for fixed $n$ but is not anytime-valid. In Appendix \ref{ss:more-simulations-gaussian}, we present some representative simulated curves of these processes.

Each $E_{\mu}(n)$ is log-normally distributed; hence we have a valid test by using the thresholds $T_{\alpha}(\mathcal{E}_0)$, $T_{\alpha}(\mathcal{E}_{\rm U})$, $T_{\alpha}(\mathcal{E}_{\rm LUS})$ or $T_{\alpha}(\mathcal{E}_{\rm LN})$. Since $\mathcal{E}_0 \supseteq \mathcal{E}_{\rm U} \supseteq \mathcal{E}_{\rm LUS} \supseteq \mathcal{E}_{\rm LN}$, we have $
T_{\alpha}(\mathcal{E}_0) \geq T_{\alpha}(\mathcal{E}_{\rm U}) \geq T_{\alpha}(\mathcal{E}_{\rm LUS}) \geq T_{\alpha}(\mathcal{E}_{\rm LN})$, so using $T_{\alpha}(\mathcal{E}_{\rm LN})$ will yield the higher power out of the four sets of distributional hypotheses. Further, since $\{E_\mu : \mu > 0\}$ is a collection of comonotonic and log-normal e-variables, Lemma \ref{lemma:comonotonic-survival} states that we can also use the threshold of $T_{\alpha}(\mathcal{E}_{\rm LN})$ on the supremum of the collection of e-variables. Finally, one may show that $E_{\Phi}$ has a decreasing density, so we may use the thresholds $T_{\alpha}(\mathcal{E}_{\rm 0})$ or $T_{\alpha}(\mathcal{E}_{\rm D})$ to control the type-I error at $\alpha \in (0, 1)$. For this simulation study, we will use $\alpha = 0.05$; we recall that numerical values for the improved thresholds are provided in Table \ref{tab:improved-threshold}. 

\begin{table}[ht]
	\centering
	\caption{Left-hand side: simulated rejection rates for tests up to $n$ observations with type-I error control of $\alpha = 0.05$. Right-hand side: number of observations required to attain a simulated rejection rate of $\beta$ with a type-I error control of 0.05. By default, tests are based on a fixed $n$, while optional stopping (OS) tests consider the rolling maximum of the e-process.}\label{tab:rejection-rates-alt}
\begin{tabular}{rr|rrrr|rrrr}
	                         &                                               &  \multicolumn{4}{c}{$n$}  & \multicolumn{4}{c}{$\beta$} \\
	                    Test &                                     Threshold &   10 &   50 &  100 &  500 & 0.5 & 0.9 & 0.95 &     0.99 \\
	                    	                    	\cmidrule(lr){1-1} \cmidrule(lr){2-2}	    \cmidrule(lr){3-6} \cmidrule(lr){7-10}
	                         &       $\mathcal{T}_{\alpha}(\mathcal{E}_{0})$ &    0 & 0.24 & 0.69 &    1 &  75 & 153 &  183 &      258 \\
	              Likelihood &   $\mathcal{T}_{\alpha}(\mathcal{E}_{\rm U})$ &    0 & 0.41 & 0.80 &    1 &  58 & 129 &  160 &      234 \\
	                   ratio & $\mathcal{T}_{\alpha}(\mathcal{E}_{\rm LUS})$ & 0.01 & 0.49 & 0.84 &    1 &  51 & 120 &  150 &      226 \\
	             $\mu = 0.2$ &  $\mathcal{T}_{\alpha}(\mathcal{E}_{\rm LN})$ & 0.07 & 0.68 & 0.91 &    1 &  34 &  96 &  123 &      194 \\
	                         &    OS+$\mathcal{T}_{\alpha}(\mathcal{E}_{0})$ &    0 & 0.29 & 0.77 &    1 &  67 & 135 &  162 &      229 \\ \hline
	                         &       $\mathcal{T}_{\alpha}(\mathcal{E}_{0})$ &    0 & 0.36 & 0.69 &    1 &  67 & 179 &  227 &      361 \\
	              Likelihood &   $\mathcal{T}_{\alpha}(\mathcal{E}_{\rm U})$ & 0.03 & 0.49 & 0.77 &    1 &  52 & 158 &  206 &      321 \\
	                   ratio & $\mathcal{T}_{\alpha}(\mathcal{E}_{\rm LUS})$ & 0.05 & 0.54 & 0.80 &    1 &  45 & 145 &  193 &      312 \\
	             $\mu = 0.3$ &  $\mathcal{T}_{\alpha}(\mathcal{E}_{\rm LN})$ & 0.17 & 0.67 & 0.86 &    1 &  31 & 124 &  171 &      286 \\
	                         &    OS+$\mathcal{T}_{\alpha}(\mathcal{E}_{0})$ &    0 & 0.46 & 0.80 &    1 &  54 & 138 &  177 &      272 \\ \hline
	                         &       $\mathcal{T}_{\alpha}(\mathcal{E}_{0})$ & 0.02 & 0.36 & 0.59 & 0.97 &  75 & 283 &  391 &      681 \\
	              Likelihood &   $\mathcal{T}_{\alpha}(\mathcal{E}_{\rm U})$ & 0.07 & 0.46 & 0.66 & 0.98 &  59 & 257 &  362 &      647 \\
	                   ratio & $\mathcal{T}_{\alpha}(\mathcal{E}_{\rm LUS})$ & 0.10 & 0.50 & 0.69 & 0.98 &  51 & 247 &  347 &      633 \\
	             $\mu = 0.4$ &  $\mathcal{T}_{\alpha}(\mathcal{E}_{\rm LN})$ & 0.23 & 0.59 & 0.75 & 0.98 &  34 & 219 &  322 &      604 \\
	                         &    OS+$\mathcal{T}_{\alpha}(\mathcal{E}_{0})$ & 0.03 & 0.49 & 0.75 & 0.99 &  51 & 179 &  246 &      451 \\ \hline
	              \multirow{4}{*}{Supremum}           &       $\mathcal{T}_{\alpha}(\mathcal{E}_{0})$ & 0.07 & 0.37 & 0.71 &    1 &  67 & 153 &  183 &      253 \\
	                 &   $\mathcal{T}_{\alpha}(\mathcal{E}_{\rm U})$ & 0.12 & 0.49 & 0.80 &    1 &  52 & 128 &  159 &      221 \\
	              & $\mathcal{T}_{\alpha}(\mathcal{E}_{\rm LUS})$ & 0.15 & 0.54 & 0.84 &    1 &  45 & 120 &  147 &      207 \\
	                         &  $\mathcal{T}_{\alpha}(\mathcal{E}_{\rm LN})$ & 0.24 & 0.68 & 0.91 &    1 &  31 &  95 &  119 &      174 \\ \hline
	\multirow{3}{*}{Mixture} &       $\mathcal{T}_{\alpha}(\mathcal{E}_{0})$ & 0.02 & 0.14 & 0.39 &    1 & 122 & 240 &  280 &      371 \\
	                         &   $\mathcal{T}_{\alpha}(\mathcal{E}_{\rm U})$ & 0.03 & 0.20 & 0.48 &    1 & 104 & 217 &  254 &      338 \\
	                         &    OS+$\mathcal{T}_{\alpha}(\mathcal{E}_{0})$ & 0.04 & 0.25 & 0.53 &    1 &  95 & 211 &  249 &      333
\end{tabular}
\end{table}

For some values of fixed $n$, 
Table \ref{tab:rejection-rates-alt} presents simulated rejection rates for 10,000 replications with type-I error control of $\alpha = 0.05$ when the simulated data come from $\mathrm{N}(0.3, 1)$. It also presents the number of observations required to obtain a rejection rate of $\beta$ for $\beta \in \{0.5, 0.9, 0.95, 0.99\}$ (based on simulated data). 

As expected, decreasing the threshold always improves the rejection rate. The improvement in rejection rate is considerable for low values of $n$ when there is little evidence from data. It is usually the case that using improved thresholds provides more power than using the optional stopping property of e-processes (although the optional stopping property also allows one to continue gathering data, which improved thresholds do not). Taking the supremum of comonotonic e-variables with the threshold $T_{\alpha}(\mathcal{E}_{\rm LN})$ leads to the highest rejection rates among every test considered.

In Appendix \ref{ss:more-simulations-gaussian}, we provide a table showing simulated rejection rates when data come from the null. The type-I error is controlled at 0.05 for all tests. The test based on the supremum process using the improved threshold of $\mathcal{T}_{\alpha}(\mathcal{E}_{\rm LN})$ achieves a simulated type-I error equal to the target level, implying that it is not wasteful at this level.

In Appendix \ref{ss:gamma-como}, we conduct a similar analysis for testing between two Gamma distributions, using the maximum likelihood e-values introduced in Section \ref{sec:43}.

\subsection{Improved threshold for universal inference}\label{ss:ui-simulation}

We next investigate the power increase we obtain by improving the threshold for universal inference e-tests. As in \cite{WRB20}, we test the null hypothesis ${\rm N}(0, 1)$, while the alternative hypothesis is that data come from a mixture of two Gaussians, that is, $w{\rm N}(\mu_1, \sigma_1) + (1-w){\rm N}(\mu_2, \sigma_2)$ for unknown $(w, \mu_1, \mu_2, \sigma_1, \sigma_2)$.
We construct the e-test using the universal inference e-variable, using 200 observations to estimate the parameters under the alternative and 200 observations to evaluate the likelihood ratio. In Figure \ref{fig:ui-mx-hist}, we present the histogram for 10,000 e-values under the null along with their log-transforms. We observe that the e-values have a decreasing density; it is reasonable to assume that the corresponding e-variable belongs to $\mathcal{E}_{\rm D}$. Also, since the density of positive log-transformed e-values is decreasing, we assume that e-variables belong to $\mathcal{E}_{\rm LD>0}$.
\begin{figure}
    \centering

    \includegraphics[width = \textwidth]{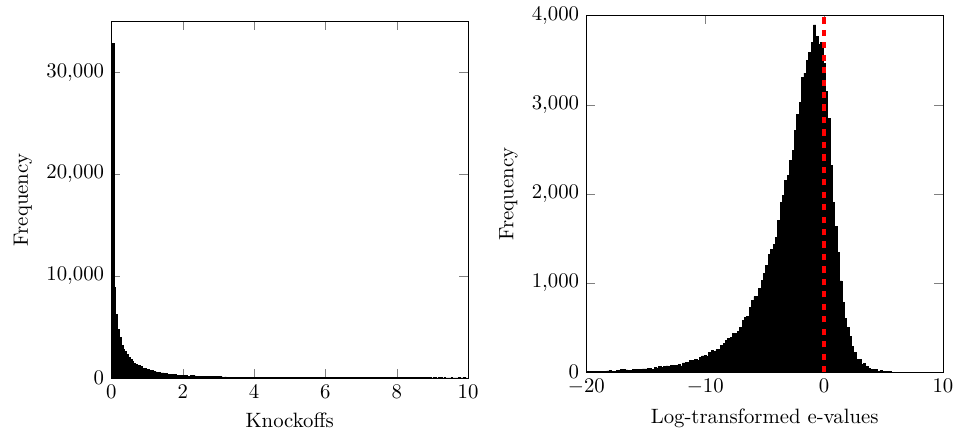}
    

    

    \caption{Histogram of null universal inference e-values and log e-values in the Gaussian mixture example of Section \ref{ss:ui-simulation}.}
    \label{fig:ui-mx-hist}
\end{figure}

In the following, we set the true model to $0.5{\rm N}(-\mu, 1) + 0.5{\rm N}(\mu, 1)$, where $\mu$ is a signal parameter. We control the type-I error at level $\alpha = 0.1$ using the thresholds $T_{0.1}(\mathcal{E}_0) = 10$, $T_{0.1}(\mathcal{E}_{\rm D}) = 5$ and $T_{0.1}(\mathcal{E}_{\rm LD>0}) = 4.07$. We simulate rejection rates from 10,000 replications and present rejection curves in Figure \ref{fig:ui-power-curve}.
\begin{figure}[ht]
	\centering

        \includegraphics{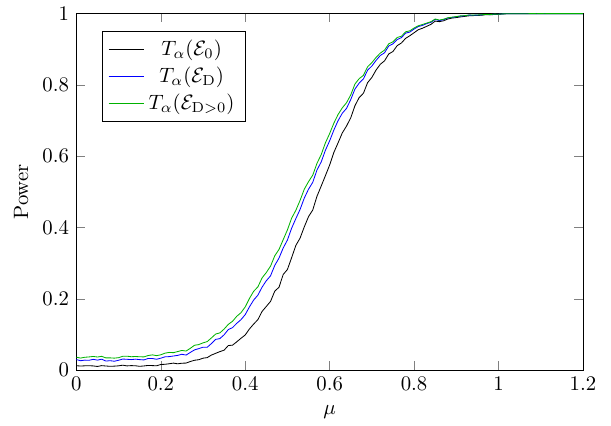}
    
			
   
                
                
	\caption{Power for universal inference e-tests with and without improved thresholds.}\label{fig:ui-power-curve}
\end{figure}

The rejection rate for $\mu = 0$ is 0.0126 using a threshold of $T_{0.1}(\mathcal{E}_{0})$, 0.0299 using a threshold of $T_{0.1}(\mathcal{E}_{\rm D})$, and 0.0365 for a threshold of $T_{0.1}(\mathcal{E}_{\rm LD>0})$. For $\mu = 0.5$, we reject the null hypothesis 30\% more often using the threshold $T_{0.1}(\mathcal{E}_{\rm D})$ compared with $T_{0.1}(\mathcal{E}_{0})$, while the rejection rate is 40\% higher when using the threshold $T_{0.1}(\mathcal{E}_{\rm LD>0})$ compared with $T_{0.1}(\mathcal{E}_{0})$.


In Appendix \ref{app:ui-simulation} we conduct simulations for the alternative hypothesis $ 0.5{\rm N}(\mu_1, 1) + 0.5{\rm N}(\mu_2, 1)$ with unknown $(\mu_1, \mu_2)$, corresponding to the e-values in Figure \ref{fig:ui-e-values} in the Introduction. The power improvement in that simulation study is similar to those discussed in this section.

\subsection{Multiple testing with boosted e-values}

Next, we illustrate the boosting factors of the e-BH procedure in a simple setting. Suppose that each e-variable under the null is $\rm Exp(1)$ distributed such that we can use the boosting factors derived in Section \ref{sec:boosting}. Further, assume that each e-variable follows a gamma distribution if the null hypothesis is not true. We set the non-null distribution's shape parameter to $1 + \Theta$, and its rate parameter to $1/(1 + \Theta)$. To incorporate uncertainty into the alternative e-variable, we model $\Theta$ using an exponential distribution with a mean of $b$ for $b \geq 0$. In other words, the e-variable under the alternative is given by the model
$$(E \vert \Theta = \theta) \lawis {\rm Gamma}(1 + \theta, 1/(1 + \theta)); \quad \Theta \lawis {\rm Exp}(1/b).$$
If $b = 0$ (i.e., $\Theta=0$), then there is no signal. We have $\E[E] = \E[(1 + \Theta)^2] = 1 + 2b + 2b^2$. 

We construct a set of $K$ e-values $(e_1, \dots, e_K)$ by simulating $K_0$ e-values from the null distribution and $K - K_0$ e-values from the alternative, where we set $K = 1000$ and $K_0 = 500$. We also construct $(e_1, \dots, e_K)$ such that its dependence structure is given by a $K$-variate Gaussian copula where off-diagonal entries of the correlation matrix are set to $-1/(K-1)$. For this dependence model, the assumption of PRDS does not hold. 

We compare the base e-BH (Table \ref{tab:e-BH-naive}) and the boosted e-BH with the knowledge that null e-variables are log-concave (Table \ref{tab:e-BH-boosted-LD}) to understand how many more discoveries we can make by using the results from Section \ref{sec:boosting}. Each number represents the average of 1000 replications of the experiment. The base e-BH only makes discoveries for $\alpha = 0.1$ and $b = 5$; even in this case, the number of discoveries is very low. On the other hand, boosting the e-values by $c_1^{\rm AD}(\alpha)$ makes discoveries when the signal or the acceptable FDR bound is high enough. Note that the realized FDP in Table \ref{tab:e-BH-boosted-LD} is still very low: the e-BH has a theoretical FDR guarantee $\alpha K_0/K$ (which is $\alpha/2$ in this example), but we observe a realized FDP much smaller. 
This observation on the conservative nature of the e-BH procedure is well known, as e-BH only has sharp FDR when the e-variables take certain two-point distributions (far from our setting); see \cite{WR22} and \cite{RB24} for discussions. 
For further comparison, simulation results on the case of using the full distribution of null e-variables and the corresponding p-BH procedure are given in Appendix \ref{app:eBH}.

{
	\begin{table}[t]
		\centering
  \caption{Number of discoveries and realized FDP as percentages based on 1000 repetitions.}
    \begin{subtable}[b]{1\textwidth}
		\caption{The base e-BH procedure}\label{tab:e-BH-naive}
		\centering
		\begin{tabular}{rrrrrrr}
			              & \multicolumn{2}{c}{$b = 3$} & \multicolumn{2}{c}{$b = 4$} & \multicolumn{2}{c}{$b = 5$} \\ 				\cmidrule(lr){2-3}\cmidrule(lr){4-5}\cmidrule(lr){6-7} 
			$\alpha$ (\%) & Discov. &          FDP (\%) & Discov. &          FDP (\%) &  Discov. &         FDP (\%) \\ \hline
			            1 &       0 &                 0 &       0 &                 0 &        0 &                0 \\
			            2 &       0 &                 0 &       0 &                 0 &        0 &                0 \\
			            5 &       0 &                 0 &       0 &                 0 &        0 &                0 \\
			           10 &       0 &                 0 &       0 &                 0 & 6.91 &                0\\\hline
		\end{tabular} 
  \end{subtable}
\\\vspace{0.3cm}
  \begin{subtable}[b]{1\textwidth}
		\caption{The boosted e-BH procedure with the knowledge of a log-concave survival function}\label{tab:e-BH-boosted-LD}
		\centering 
		\begin{tabular}{rrrrrrr}
			              & \multicolumn{2}{c}{$b = 3$} & \multicolumn{2}{c}{$b = 4$} & \multicolumn{2}{c}{$b = 5$} \\ 			 \cmidrule(lr){2-3}\cmidrule(lr){4-5}\cmidrule(lr){6-7} 
			$\alpha$ (\%) & Discov. &          FDP (\%) & Discov. &          FDP (\%) & Discov. &          FDP (\%) \\ \hline
			            1 &       0 &                 0 &   65.91 &                 0 &  208.70 &                 0 \\
			            2 &       0 &                 0 &  133.38 &                 0 &  233.66 &                 0 \\
			            5 &    5.55 &                 0 &  196.16 &                 0 &  269.32 &                 0 \\
			           10 &   65.11 &                 0 &  234.48 &                 0 &  294.65 &             0.001 \\ \hline
		\end{tabular}
  
  \end{subtable}
	\end{table}
}

\section{Conclusion}\label{sec:concl}

We have investigated methods to enhance the power of e-tests in situations where we have some knowledge of the distribution of the e-variables. Distributional assumptions such as decreasing or unimodal densities can often be numerically verified or theoretically justified. In most relevant cases, the e-value threshold can be roughly improved by a factor of $2$ or $e$. We also show that when a collection of e-variables is comonotonic, we can take the supremum over that collection and obtain a valid test. This enables us to improve upon split likelihood ratio tests and use the entire dataset to estimate and test. In simulation experiments, the various improvements presented in this paper result in visibly more powerful tests than the standard methods. 

\section*{Acknowledgements}

The authors thank the AE and the two reviewers for their helpful suggestions. The authors acknowledge financial support from the Natural Sciences and Engineering Research Council of Canada (CBW: PDF-578273-2023 and RGPIN-2025-06879; RW: RGPIN-2024-03728 and CRC-2022-00141).


\newpage

\begin{center}
    \Large SUPPLEMENTARY MATERIAL
\end{center}
\begin{appendix}

    \section{Construction of e-values in the introduction}\label{app:constr-e-intro}

    The null knockoffs in Figure \ref{fig:knockoff-e-values} are constructed similarly to the simulation study in Section 4 of \cite{RB24}. To generate one e-value, we use a Gaussian linear model with 100 observations and 80 variables. The number of non-null variables is either 0 or 8. The joint distribution of the feature vector $\boldsymbol{X}$ is Gaussian with mean 0 and covariance matrix $\Sigma$ with entries $\Sigma_{ij} = 0.5^{|i - j|}$. The distribution of $Y \vert \boldsymbol{X}$ is given by a linear model with mean $\boldsymbol{X}^{\rm T} \boldsymbol{\beta}$, where $\boldsymbol{\beta}$ is the vector of regression coefficients. In the case of zero non-nulls, we fix the vector $\boldsymbol{\beta} = \boldsymbol{0}$, a vector with 80 zeroes. When there are eight non-nulls, $\boldsymbol{\beta}$ contains eight values of one, which we distribute uniformly on the vector, with the remaining 72 entries set to zero. We construct knockoffs using the lasso coefﬁcient difference statistic, with a target FDR level of 0.05. We draw 50 copies of knockoffs and report the average. 
    
    We construct the null universal inference e-values from Figure \ref{fig:ui-e-values} similarly to the mixture model example in Section 3 of \cite{WRB20}. The null hypothesis is that data come from $\mathrm{N}(0, 1)$, while the alternative hypothesis is that data come from $0.5{\rm N}(\mu_1, 1) + 0.5 {\rm N}(\mu_2, 1)$, for unknown $(\mu_1, \mu_2)$. We use 400 observations: 200 to estimate the parameters under the alternative and 200 to compute the e-value from the likelihood ratio test statistic. We estimate the parameters from the alternative hypothesis using maximum likelihood estimation with the expectation-maximization algorithm.
    
    \section{Proof of Theorem \ref{th:dec-uni}, parts (ii) and (iii)}
    
    We now prove parts (ii) and (iii) of Theorem \ref{th:dec-uni}. The proof arguments are similar to part (i) but more complicated. 
    
    \begin{enumerate}
        \item[(ii)] Take $E \in \mathcal{E}_{\rm D > 1}$ and $z \ge 1$. Let $f$ be the density function of $E$. 
    We construct an e-variable $E_0$ with density $g_0$ such that 
    (a) $g_0(x) = f(z)$ for $1 \le x < z$;
    (b) $g_0(x) = f(x)$ for $x \ge z$; 
    (c) $E_0$ has a point mass at 0 with probability 
    $1 - \int_1^\infty g_0(x) \d x$. 
    Clearly, 
    $\E[E_0] \le \E[E]$ 
    and 
    $\p(E_0 \ge z) = \p(E \ge z)$. 
    Construct another e-variable $E_1$ with density $g_1$ such that 
    $g_1(x) = g_0(x)$ for $0\le x < z$ and $g_1(x) = f(z)$ for $z \leq x \leq z + \p(E \ge z)/f(z)$. Since $f$ is decreasing for $x > z$, we have 
    $\E[E_1] \le \E[E_0]$ 
    and 
    $\p(E_1\ge z) = \p(E_0 \ge z)$.
    
    Therefore, for any $E \in \mathcal{E}_{\rm D>1}$, there exists an e-variable $E' \in \mathcal{E}_{\rm D>1}$ such that $E'$ has a point mass at 0 and a uniform density on $[1, b]$ for some constant $b$, and $\p(E'\ge z) = \p(E \ge z)$. To show $\p(E \ge z) \leq z - \sqrt{z^2 - 1}$, we can consider the collection of e-variables that have a point mass at 0 and a uniform density on $[1, b]$ for any $b > 1$.
    Moreover, it suffices to consider the case $\E[E]=1$.
    
    Denote by $c$ the constant density of $E'$ over $[1, b]$ and note that $\E[E'] = \int_1^b cx \d x = c(b^2 - 1)/2$. Therefore, $\E[E'] = 1$ implies $c = 2/(b^2 - 1)$. We have $\p(E'\ge z) = 0$ if $z > b$, and otherwise 
    \begin{equation}\label{eq:decreasing-prob1}
        \p(E'\ge z) = \int_z^b c \d x= \int_z^b \frac{2}{b^2 - 1}\d x = 2\frac{b - z}{b^2 - 1}.
    \end{equation}
    Let us now find the constant $b$ maximizing the right-hand side of \eqref{eq:decreasing-prob1}. Taking the derivative with respect to $b$ and setting the result equal to zero, we find that $b$ is the solution to $b^2 - 2zb + 1 = 0$, or $b = z + \sqrt{z^2 - 1}$ (we take the positive solution to ensure $b > z$). Inserting the latter into the right hand side of \eqref{eq:decreasing-prob1}, we have $$\p(E'\ge z) \le \frac{\sqrt{z^2 - 1}}{z^2 + z\sqrt{z^2 - 1} - 1} = z - \sqrt{z^2 - 1},$$
    which, together with the attainability of this distribution, shows     $R_\gamma(\mathcal{E}_{\rm D>1}) = 1/\gamma  - \sqrt{1/\gamma^2 - 1}.$  The last equality is standard algebra.  
        \item[(iii)] We first show $ \p(E\ge 1/\gamma )\le (\gamma /2)\vee (2\gamma -1)$ for each $E\in \mathcal E_{\rm U}$. 
        An alternate proof to a similar result can be found in Theorem 5 of \cite{bernard2023corrigendum}.
    Let $m $ be the largest mode of $E$ and 
    $f$ be the density function of $E$.
    Denote by  $z=1/\gamma$. 
    We consider two cases separately.  
    
     First, consider the case that $z  \le m $.
      In this case, we construct an e-variable $E_0$ with density $g_0$ such that $g_0=\int_0^z f(x) \d x /z$ and $E_0$ has a point mass at $z$ with probability $1-\int_0^z f(x) \d x$. 
      Since $f$ is increasing on $[0,z)$, replacing $f$ with $g_0$ on $[0,z)$ will decrease the mean of the distribution. Moreover, moving all probability mass of $E$ on $[z,\infty)$ to the point $z$ will decrease the mean. Hence, $\E[E_0]\le \E[E]$. 
      Note that $\p(E_0\ge z)=\p(E\ge z)$, and $E_0\in \mathcal E_{\rm U}$.
    
    Second, consider the case that $z > m $.
    In this case, 
    consider another e-variable $E_1$ with density function $g_1$ such that
    (a) $g_1(x)=f(x)$ for $x\ge z$;
    (b) $g_1(x)= f(z)=:a $ for $x\in (m ,z]$;
    (c) $g_1(x) =  \int_0^m f(x)/m$ for $x\in [0,m)$; 
    (d) $E_1$ has a point mass  at $m$ with probability $p:=1-\int_0^\infty g_1(x) \d x$.  
    Clearly, $\p(E\ge z)=\p(E_1\ge z)$. We will verify that $\E[E_1]\le \E[E ] \le 1$.
    This is consistent with the argument presented above. 
    Note that 
    on  $x\in [0,m)$, replacing the density $f$ with the constant density $g_1$ will decrease the mean,
    since $f$ is an increasing function on $[0,m)$. 
    Moreover, on $[m,z)$, replacing $f$ with the point mass and $g_1$ will also decrease the mean. 
    Putting both observations together, we have $\E [E_1]\le \E[E]$.
    Further, $E_1$ also has a unimodal density.
    
    Next, construct e-variable $E_2$ with density function $g_2$ such that 
      (a) $g_2(x)=g_1(x)$ for $x> m$;
    (b) $g_2(x)= g_1(x) +p/m=:b $ for $x\in [0,m)$.
    In other words, the point mass at $m$ of $E_1$ is replaced by a uniform density on $[0,m)$. 
    Certainly, $\E[E_2]\le \E[E_1]\le1$. 
    By construction, $\p(E_1\ge z)=\p(E_2\ge z)$. 
    
    If either $m=0$ or $b\ge a$, then $E_2$ has a decreasing density. 
    Hence, using part (i) of Theorem \ref{th:dec-uni}, we have $\p(E_2\ge z) \le 1/(2z)$. 
    This gives   $\p(E\ge z)= \p(E_2\ge z)\le1/(2z)$.
    
    Next, assume $b<a$. In this case, construct e-variable $E_3$ with density function $g_3$ such that 
      (a) $g_3(x)=g_2(x)$ for $x\ge z$;
    (b) $g_3(x)=   \int_0^{z} g_2(x) /z =:d $ for $x\in [0,z)$.
    Using the same argument above, we have $\E[E_3]\le \E[E_2]$, and $E_3$ has a unimodal density with mode at $z$. By construction, $\p(E_2\ge z)=\p(E_3\ge z)$. 
    
    Construct e-variable $E_4$ with density function $g_4$ such that $g_4(x)=g_3(x)$ for $x\in [0,z)$ and $E_4$ has a point mass  at $z$ with probability $ \int_{z}^\infty g_3(x)\d x$.
    We have $\E[E_4]\le \E[E_3]$ and $E_4\in \mathcal E_{\rm U}$ and by construction, $\p(E_3\ge z)=\p(E_4\ge z) $.

    The above arguments in two cases together show that for any $E\in \mathcal E_{\rm U}$, either  $\p(E\ge z) \le1/(2z)$ 
    or there exists $E'\in \mathcal E_{\rm U}$ such that $E'$ has a uniform density on $[0,z)$ and a point mass at $z$, and $\p(E'\ge z)=\p(E\ge z)$. This $E'$ is $E_0$ in the case  $z\le m$ and $E_4$ in the case $z>m$.
     It suffices to analyze $q:=\p(E'\ge z)$. Note that $\E[E']=zq + (1-q)z/2\le 1$, giving rise to $q \le 2/z-1$. Hence, in this case, $\p(E\ge 1/\gamma) =\p(E'\ge 1/\gamma) \le 2\gamma -1$.
    Putting the above two cases together we have $ \p(E\ge 1/\gamma ) \le (\gamma/2) \vee (2\gamma -1). $
    
    To show the converse statement $R_{\gamma}(\mathcal E_{\rm U}) \ge (\gamma/2) \vee (2\gamma -1)$,
    it suffices to verify the two inequalities separately, namely, 
     $R_{\gamma}(\mathcal E_{\rm U}) \ge \gamma/2 $ for $\gamma \in (0,2/3]$ and
       $R_{\gamma}(\mathcal E_{\rm U}) \ge 2\gamma -1 $ for $\gamma \in (2/3,1]$.
       The first inequality follows from part (i) of Theorem \ref{th:dec-uni}, which gives
    $R_{\gamma}(\mathcal E_{\rm U}) \ge R_{\gamma}(\mathcal E_{\rm D}) =  \gamma/2 $ for $\gamma \in (0,1)$. 
    To show the second inequality, for $z\in [1,2]$,  
    construct $E$ with a uniform density on $[0,z)$ and a point mass at $z$ and with probability $q:=2/z-1$, we get $\E[E]=zq + (1-q)z/2 = 1$. Hence, $E\in \mathcal E_{\rm U}$, and by choosing $ \gamma=1/z\in [1/2,1]$, we get $\p(E\ge 1/\gamma) = 2\gamma -1$, showing the second inequality. 
    \end{enumerate}

\section{Additional results for Section 3}\label{sup:moment-constraints}

In the main part of the paper, we have mainly focused on improving thresholds based on shape constraints of the density function of the e-variables. If the statistician has other information, such as moment constraints on e-variables, they may incorporate this information to improve the threshold. In this appendix, we provide some improved thresholds when including variance information. 

Define the set of e-variables with a bounded variance: 
\begin{align*}
\mathcal E_{\sigma^2}&=\{E\in\mathcal E_0: \mathrm{var}(E) \leq \sigma^2\}.
\end{align*}
The following theorem provides improved thresholds for e-variables with variance constraints.

\begin{theorem}\label{thm:improved-threshold-variance}
  For $\gamma \in (0, 1]$, we have
    $$R_\gamma(\mathcal{E}_{\sigma^2}) = \left(\frac{\sigma^2}{\sigma^2 + (1/\gamma - 1)^2}\right) \wedge \gamma.$$
\end{theorem}

\begin{proof}
  We first show that, for every $x > 1$,
  
  \begin{equation}\label{eq:variance-threshold}
    \sup_{\substack{X \geq 0, \, \E[X]\leq 1, \\ \mathrm{var}(X) \leq \sigma^2}}
  \Pr(X\ge x)
  =\min \left\{\frac 1 x, \, \frac{\sigma^2}{\sigma^2+(x-1)^2}\right\}.
  \end{equation}
  The expression in \eqref{eq:variance-threshold} holds as an upper bound: the first part of the bound is due to Markov's inequality, while the second is due to the Cantelli inequality (also called the one-sided Chebyshev inequality). For $t > 0$, the one-sided Cantelli bound is given by 
  $$\p(X - \mu \geq t) \leq \frac{\sigma^2}{\sigma^2 + t^2}.$$
  Setting $t = x - \mu$ and noting that $x - \mu \geq x - 1$, we have for $x > 1$ that
  $$\p(X \geq x) = \p(X - \mu \geq x- \mu) \leq \frac{\sigma^2}{\sigma^2 + (x - \mu)^2} \leq \frac{\sigma^2}{\sigma^2 + (x - 1)^2}.$$

  We now show that \eqref{eq:variance-threshold} holds as a lower bound. 
  \begin{itemize}
    \item Case 1: If $\sigma^2 + 1 > x$, \eqref{eq:variance-threshold} is attained through Markov's inequality. That is, at the two-point distribution with $\p(X = 0) = 1-1/x$ and $\p(X = x) = 1/x$. In that case, $\mathrm{var}(X) = x^2/x - 1 = x - 1$ and the variance constraint does not factor into the bound.
    \item Case 2: If $\sigma^2 + 1 \leq x$, consider a two-point distribution with 
    $$\p(X = x) = p : =\frac{\sigma^2}{\sigma^2 + (x - 1)^2}$$
    and
    $$\p\left(X = \frac{1 - px}{1-p}\right) = 1 - p.$$  Then, we have $E[X] = 1$ and $\mathrm{var}(X) = \sigma^2$
    and $\p(X \geq x) = \sigma^2/(\sigma^2 + (x - 1)^2)$. Note that $(1 - px)/(1-p)\geq 0$ whenever $\sigma^2 \leq x-1$, so for case 2 the resulting distribution is a positive random variable. 
  \end{itemize}
\end{proof}

\begin{remark}\label{rem:unimodal-variance}
We can also incorporate shape constraints along with variance constraints. Define 
    \begin{align*}
        \mathcal E_{\rm U, \sigma^2}&=\{E\in\mathcal E_0: \mbox{$E$ has a unimodal density on $[0, \infty)$ and } \mathrm{var}(E) \leq \sigma^2\},
    \end{align*}
    noting that $\mathcal E_{\rm U, \sigma^2} = \mathcal E_{\rm U} \cap \mathcal E_{\sigma^2}$. We can derive the worst-case bound $R_{\gamma}(\mathcal E_{\rm U, \sigma^2})$ using Theorem 3 of \cite{bernard2023corrigendum}; we omit its expression since it is lengthy and tedious to write out in full.  

\end{remark}

In Figure \ref{fig:threshold-comparison-variance}, we plot the worst-case type-I error bounds $R_\gamma(\mathcal E)$ and the associated rejection thresholds $T_\alpha(\mathcal E)$ for the classes of e-variables subject to a variance bound $\sigma^2$, both with and without the additional assumption of unimodality. For comparison, we overlay the corresponding curves for the unrestricted class $\mathcal E_0$ and the unimodal class $\mathcal E_{\rm U}$.  We see that, for small $\gamma$, enforcing a variance constraint yields tighter thresholds than the Markov bound; as $\gamma$ grows, these ``variance-only'' thresholds approach those of $\mathcal E_0$, coinciding exactly once $\gamma>1/(1+\sigma^2)$.  Likewise, the thresholds for the variance-and-unimodal class $\mathcal E_{\rm U,\sigma^2}$ match those of $\mathcal E_{\rm U}$ for larger $\gamma$.

\begin{figure}[ht]
    \centering

    \includegraphics[width = \textwidth]{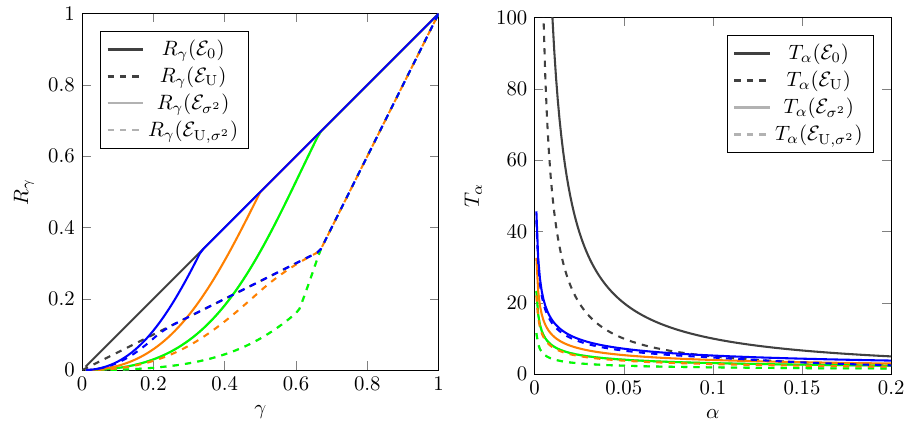}

\caption{Comparison of worst-case type-I errors and improved thresholds for 
variance constraints (full lines without unimodality constraints, dashed lines with unimodality constraints). Green: $\sigma^2 \leq 0.5$, orange: $\sigma^2 \leq 1$, and blue: $\sigma^2 \leq 2$. 
}   \label{fig:threshold-comparison-variance}
\end{figure}

  
  

The analysis from this appendix can be further extended by considering additional moments. In particular, Theorem 3.2 of \cite{bertsimas2005optimala} provides a method to obtain tight bounds on the tail probability of non-negative random variables with moment conditions.

    \section{Proof of Theorem \ref{th:log-transformed}}\label{app:proof-log-transformed}
    
    In this section, we provide proofs for the five results in Theorem \ref{th:log-transformed}.
    
    \begin{enumerate}
        \item[(i)] Let $x\ge \log 2$,
        $a>0$ and $p<e^{-x}\le 1/2$.
        Consider a three point distribution at $x$ with probability $p$, at $x-a$ with probability $1-2p$
        and at $x-2a$ with probability $p$.
        Clearly, this distribution is symmetric, and we let $X$ follow this distribution.
        We can compute 
        $\E[e^X] = p e^{x} + p e^{x-2a} + (1-2p)e^{x-a}$.
        Sending $a\to \infty$ gives 
        $\E[e^X]\to p e^{x}<1 $. Therefore, $e^X\in \mathcal E_{\rm LS}$ for sufficiently large $a$.
        Note that $\p(e^X\ge e^x) = p$.
        Since $p$ is an arbitrary number smaller than $e^{-x}$, we get $R_{\gamma} (\mathcal E_{\rm LS}) \ge  \gamma$ for all $\gamma \le 1/2$, and it holds as an equality due to Markov's inequality. 
        Next, suppose that $X$ has any symmetric distribution and $\E[e^X]\le 1$.
        Jensen's inequality gives 
        $\E[e^X]\ge e^{\E[X]}$.
        Hence, $\E[X]\le 0$. 
        This implies $\p(e^X>1)=\p(X>0)\le 1/2$. This gives 
        $R_\gamma (\mathcal E_{\rm LS}   ) =\gamma \wedge (1/2)$. The last statement $R_1(\mathcal E_{\rm LS}  ) =1$ is easily verified by taking $X=0$.
        \item[(ii)] Let $m>0$ be large, $z=-\log \gamma$ and $p\in (0,\gamma)$, and let $X$ follow the distribution given by combining a uniform distribution on $[z-m,z]$ with density $p/m$
        and a point mass at $z$ of probability $1-p$.
        This distribution is unimodal, and $\E[e^X]$ can be easily computed as 
        $$
        \E[e^X] = (1-p) e^z + \frac p m  \int_{z-m}^z e^x \d x
        = (1-p)\frac 1{\gamma} + 
        \frac pm (e^{z}- e^{z-m}).
        $$
        It is clear that for $\E[e^X]\le 1 $ by sending $m\to\infty$ $\E[e^X]<1$.
        Note that
        $\p(e^X\ge 1/\gamma) = p$ and $p$ can be any number smaller than $\gamma$. This shows 
        $R_\gamma (\mathcal E_{\rm LU}   ) \ge \gamma$. The other direction $R_\gamma (\mathcal E_{\rm LU}   ) \le \gamma$ is Markov's inequality. 
        \item[(iii)] Let $E \in \mathcal E_{\rm LD>0}$ and $X= \log E$. Let $g$ denote the density function of $X$ and $f$ denote the density function of $E$. We have that
        $f(x) = g(\log x) / x$ for $x \ge 0$. Note that if $E \in \mathcal{E}_{\rm LD>0}$, then $g'(x) < 0$ for $x > 0$. Further, for $x > 0$, we have that
        $$f'(x) = g'(\log x)/x^2 - g(\log x)/x^2 < 0,$$
        implying $\mathcal{E}_{\rm LD>0} \subseteq \mathcal{E}_{\rm D>1}$, which proves the upper bound. 
        
        
        Suppose we have a random variable $Y$ with distribution function $H$. If we construct a new random variable $Y'$ with distribution function $H'$ by shifting some mass to the left, then $\p(Y' \le x) \ge \p(Y \le x)$, so $Y$ dominates $Y'$ under the usual stochastic order and $\E[\phi(Y')] \le \E[\phi(Y)]$ for any increasing function $\phi$.
        
        Let us now show that $R_\gamma(\mathcal E _{\rm LD>0}) \le e^{-a_\gamma}$. Start with a random variable $X\in \mathcal{E}_{\rm LD>0}$ with density function $g$ and let $z = 1 / \gamma$. Let $m \leq \inf({\rm supp}(X))$, that is, $m$ is smaller than the left-hand support of $X$ (if $\inf({\rm supp}(X)) = -\infty$, then we mean that $m$ approaches $-\infty$). Construct a new random variable $X_0$ with density function $g_0$ such that 
        (a) $g_0(x) = g(x)$ for $x > \log z$;
        (b) $g_0(x) = g(\log z)$ for $0 \le x \le \log z$;
        (c) $X_0$ has a point mass at $m$. Then $\p(X_0 \ge \log z) = \p(X \ge \log z)$ and $X_0$ is smaller than $X$ under the usual stochastic order since we shifted mass to the left.
        
        Construct another random variable $X_1$ with density function $g_1$ such that
        (a) $g_1(x) = g_0(x)$ for $x \le \log z$;
        (b) $g_1(x) = g(\log z)$ for $\log z \le x \le \log z + \p(X \ge \log z ) / g(\log z)$. Then $\p(X_1 \ge \log z) = \p(X_0 \ge \log z)$ and $X_1$ is smaller than $X_0$ under the usual stochastic order since we shifted mass to the left.
        
        This shows that if $E \in \mathcal{E}_{\rm LD>0}$, then we can construct a random variable $E' \in \mathcal{E}_{\rm LD>0}$  such that $\log E'$ has a point mass $m$ approaching $-\infty$ and a constant density over $[0, a]$ for some $a \ge \log z$, and $\log E'$ is smaller than $\log E$ under the usual stochastic order, implying that $\E[E'] \le \E[E]$. 
        
        Denote by $c$ the constant density of $\log E'$ over $[0, a]$ and $p = \p(\log E' = m)$. 
        We have that 
        $\E[E'] = \E[e^{\log E'}] = p e^m+c(e^a - 1)$. Since $E'$ is an e-variable and since $m$ approaches $-\infty$, we have $c \le 1/(e^a - 1)$. 
        This implies that, for any $E \in \mathcal{E}_{\rm LD>0}$, 
        \begin{equation}\label{eq:pe-ld0}
            \p(E \ge z) = \int_{\log z}^{a}c \d x \le \int_{\log z}^{a}1/(e^a - 1) \d x = \frac{a - \log z}{e^a - 1}.
        \end{equation}
        We obtain the supremum of the right-hand side of \eqref{eq:pe-ld0} over $a > \log z$, which is the solution to the equation $e^a(1 - a + \log z) = 1$.  
     Alternatively, using the substitution $t = e^a/z$, we seek the solution for $t$ in $t(1 - \log t) = 1/z$ 
     for $t > 1$, that we denote by $t_{1/z}$. 
     Note that $$\frac{\d}{\d t} (t(1-\log t)) = -\log t,$$ implying that $t \mapsto t(1 - \log t)$ is monotonically decreasing and $t_{1/z}$ has a unique solution for $t > 1$. 
     It is easy to verify that $e$ is the solution to $t(1 - \log t) = 1/z$ when $1/z \downarrow 0$, hence $t_{1/z} \in (1, e)$. Replacing $t_{1/z}$ into \eqref{eq:pe-ld0} and simplifying yields
        $\p(E \ge z) \le 1/(z t_{1/z})$.

        To show the upper bound holds with equality, fix some $z > 1$. We construct an e-variable $E$ such that $\log E$ has a constant density of $1/(t_{1/z}z - 1)$ over $[0, \log (t_{1/z}z)]$, where $t_{1/z}$ satisfies $t_{1/z}(1 - \log t) = 1/z$, and the remaining mass approaching $-\infty$. Then, $\p(E \ge z) = e^{-a_{1/z}}$.
        
        To show the lower bound, we note that $R_\gamma(\mathcal{E}_{\rm LD>0})$ maximizes $(a + \log \gamma)/(e^a - 1)$, so any $a \in (-\log \gamma, \infty)$ will give a lower bound for $R_\gamma(\mathcal{E}_{\rm LD>0})$. We start with the observation that $\gamma \mapsto a_\gamma$ is decreasing, so $e^{a_\gamma}$ is large for small $\gamma$, implying that $1 - a_\gamma - \log \gamma$ is a small positive number. We approximate $a_\gamma$ by the solution for $a$ in $1 - a - \log \gamma = 0$, which leads to $a = 1 - \log \gamma$. Replacing this in the right-hand side of \eqref{eq:pe-ld0} leads to the lower bound. 
        \item[(iv)] Let $m\ge 1$ and $c>0$. Consider the  uniform distribution $F_{m,c}$  on $[m-1/c,m]$, and let $X$ follow this distribution.
        Clearly, $F_{m,c}$  is a unimodal and symmetric distribution,
        and it has a constant, thus decreasing, density.  We can   compute
        $$\E[e^X] = \int_{m-1/c}^m c e^x \d x = c  e^m (1 -e^{-1/c}) \le ce^m.$$
        By taking $c=e^{-m}$, we know that $e^X\in \mathcal E_{\rm LUS}\cap  \mathcal E_{\rm LD}$. 
        For $x\ge 0$, we have 
        $\p(e^X\ge e^x)  = c (m-x) =e^{-m} (m-x) $.
        By taking $m=x+1$ we get
        $$\p(e^X\ge e^x) = e^{-x-1}.$$
        Since  $e^X\in \mathcal E_{\rm LUS}$, we get 
        $  R_\gamma (\mathcal E_{\rm LUS}   )\ge \gamma/e$
        and 
        $  R_\gamma (\mathcal E_{\rm LD}   )\ge \gamma/e.$
        
        Next, we show the upper bound.
        We will first argue that 
        both $  \mathcal E_{\rm LUS}$
        and 
        $\mathcal E_{\rm LD}$
        admit the same class of worst-case distributions.  
        Fix $x> 0$. For any $E\in \mathcal E_{\rm LUS}$ or $E\in \mathcal E_{\rm LD}$,
        let $X=\log E$. Denote the distribution of $X$ by $F$ and its 
        density function by $g$.
        
        First, we argue the case of $ \mathcal E_{\rm LUS}$. 
        Denote by $x_0$ the median of $F$, which is unique.  We   consider another distribution $G$,  such that $G$ has a point mass at $x_0$ and a uniform density on 
        $[x_0-a, x_0+a]$ for some $a>0$.
        Let  $g$ be the density function of $G$ outside $\{x_0\}$. 
        Let $Y$ follow the distribution $G$. 
        We choose $a$ and the point mass at $x_0$  such that $f(x)=g(x)$ and  $F(x)=G(x)$.  It is clear that such $G$ exists, and it is straightforward to check that $G$ is dominated by $F$ in increasing convex order (see e.g., Theorem 3.A.44 of \cite{SS07}), that is, $$\int_\R u \d F \ge \int_\R u \d G \mbox{~~~~ for all increasing convex functions $u:\R\to \R$}.$$  
        It follows that $\E[e^Y]\le \E[e^X]\le 1$, and thus $e^Y\in \mathcal E_{\rm LUS}$. Moreover, $\p(Y\ge x)=\p(X\ge x)$. 
        Therefore, it suffices to consider distributions of the form $G$.
        Keeping the density $g$ at $x$ fixed, we can increase $a$ and remove the point mass at $x_0$.
        This will reduce the mean of $Y$ and does not change $\p(Y\ge x)$. Therefore, it suffices to consider uniform distribution on some interval $[m-1/c,m]$, thus the distribution $F_{m,c}$ above.
        
        Next, consider the case of $\mathcal E_{\rm LD}$, which is simpler. 
        Let 
        $Y$ follow the distribution $G=F_{m,c}  $ with density $g$, such that
        $G(x) = F(x) $ and 
        $g(x)=f(x)$.
        It is straightforward to check that $F$ dominates $G$ in increasing convex order, and  $e^Y\in \mathcal E_{\rm LD}$. 
        Hence, we, again, only need to consider $F_{m,c} $ above. Moreover, $ \mathcal E_{\rm LUS}=\mathcal E_{\rm LD}$. 
        
        For $Y\lawis F_{m,c}$, note that
        $\p(Y>x) = c(m-x)_+$. To find the largest value of this probability, it is clear that $m>x$. 
        Since the median of $Y$ is less than or equal to $0$, as we argued above,
        we know $1/c\ge 2m \ge 2x$. 
        Therefore, by using $c  e^m (1 -e^{-1/c}) \le 1$, we get $c e^m (1-e^{-2x})\le 1$, and thus $c \le e^{-m} (1-e^{-2x})^{-1}$. 
        This leads to 
        $$
        \p(Y>x) = c(m-x)_+ \le \sup_{m\ge x} e^{-m} (1-e^{-2x})^{-1} ( m-x)  = \frac{ e^{-x-1}} {1-e^{-2x}}.
        $$
        This shows 
        $$
        R_{\gamma}(\mathcal E_{\rm LUS}) \le \frac{\gamma /e }{1-\gamma^2}.
        $$
        The remaining inequality  $R_{\gamma}(\mathcal E_{\rm LD}) \le R_\gamma(\mathcal{E}_{\rm LD})$ follows from $\mathcal E_{\rm LD} \subseteq \mathcal E_{\rm LD>0}\subseteq \mathcal E_{\rm D>1}$.  
        \item[(v)] Clearly, if $\gamma=1$, then $R_1(\mathcal E_{\rm LN})=1$ by choosing the trivial degenerate e-variable $E=1$. 
        Next,
        let $\sigma > 0$ and $X_\sigma$ follow $\mathrm{LN}(-\sigma^2/2 ,\sigma^2)$, which has mean $1$. 
        For $x>1$, we have
        \begin{align*}\sup_{\sigma> 0}\p(X_\sigma\ge x) &= \sup_{\sigma> 0} \Phi\left(\frac{-\log x-\sigma^2/2}{\sigma}\right) \\ &= \Phi\left(\sup_{\sigma> 0}\left( -\frac{\log x}{\sigma} -\frac{\sigma}{2}\right)\right)    =\Phi\left(-\sqrt{2 \log x}\right).
        \end{align*}
        This shows $R_\gamma (\mathcal E_{\rm LN}   ) =\Phi(-\sqrt{-2\log\gamma})$. 
    
    We can provide an explicit upper bound by using the inequality, $$
    \left(\frac{1}{x}-\frac1{x^3} \right)\phi(x)< 
    \Phi(-x) < \frac{1}{x} \phi(x),$$ for $x > 0$, where $\phi$ is the standard normal pdf (see, e.g., \cite{gordon1941values}). It follows from 
    part (v) of Theorem \ref{th:log-transformed} that for $\gamma<1$,
    $$R_\gamma (\mathcal E_{\rm LN}   ) =\Phi(-\sqrt{-2\log\gamma})\leq  \frac{\gamma}{2\sqrt{-\pi \log \gamma}}.$$
    
    \end{enumerate}
    
    \section{Proof of Theorem \ref{thm:log-distribution}}\label{app:proof-log-distribution}
    
    In this section, we provide proofs for the three results in Theorem \ref{thm:log-distribution}. 
    
    \begin{enumerate}
        \item[(i)] Let $E\in \mathcal E_{\rm LCS}$
        and denote by $G$ its survival function.
        Note that log-concavity implies for   $x,y\ge 0$ and $\theta \in [0,1]$,
        $$G(\theta x+(1-\theta )y)\geq G(x)^{\theta }G(y)^{1-\theta }.
        $$
        Moreover, log-concave survival functions cannot jump at $0$, hence $G(0)=0$.
        Taking $y=0$, $x>1$ and $\theta = z/x$ with $z\in (0,x]$,  we get
        $$
        G(z) \ge G(x)^{z/x}. 
        $$
        For fixed $x>1$, denote by $t_x =(\log G(x))/x$, which is negative by Markov's inequality. We have
        \begin{align*}
        \E[E] = \int_0 ^\infty  G(z)\d z
        \ge \int_0^x G(x)^{z/x} \d z
        = \int_0^x e^{ t_x z} \d z 
        = \frac{e^{x t_x } -1}{t_x} \le 1.
        \end{align*}
        Therefore,
        $$
        e^{x t_x}\ge t_x+1.
        $$
        Note that the above inequality always holds if $t_x\le -1$.
        Below, we consider the case $t_x>-1$. 
        We note that 
        $$
        e^{x t} -  t -1=0.
        $$ 
        has a unique solution for $t \in (-1,0)$. 
        Note that $e^{xt} - t-1> 0$ for $t=-1$
        and $e^{xt} - t-1<0$ for $t$ close to $0$. 
        Strict convexity guarantees a unique solution on $ (-1,0)$, denoted by $t_x^*$. 
        The above argument also shows $t_x\le t_x^*$,
        and hence 
        $G(x) =  e^{x t_x} \le e^{xt_x^*}$ for all $x> 1$.
        Taking $\gamma=1/x$ and $s_\gamma=t^*_x$  proves 
        $R_{\gamma}(\mathcal E_{\rm LCS} )
        \le e^{s_\gamma/\gamma}$.
        
        To show that the above inequality holds as an equality, let
        $
        G(z) = e^{t_x^* z} 
        $
        for $z< x$, and $G(z)=0$ for $z\in [0,x]$.
        It is clear that $G$ is a log-concave survival function, and the corresponding random variable $E$ has mean $1$,
        and thus $E\in \mathcal E_{\rm LCS}$. 
        For any $z\in [0,x)$,
        we have 
        $\p(E>z)  = e^{t_x^* z} $. 
        By continuity, with $\gamma=1/x$, we get 
        $R_{\gamma} (  \mathcal E_{\rm LCS}) \ge  e^{t_x^* x} = e^{s_\gamma/\gamma} $.  
        Thus, the first equality in the proposition holds.  
        
        To show the remaining inequality, by using 
         $$
        e^{-x t_x} (t_x+1) \le 1,
        $$
        and $t_x\in (-1,0)$,  we get 
         $(1 - xt_x)(t_x+1) \le 1,$ 
        which is
        $ 
        -x t_x^2 +t_x -x t_x \le 0,
        $ 
        leading to 
        $ 
        -x t_x \ge x-1.
        $ 
        Thus
        $ 
        t_x \le  (1-x)/x, 
        $ 
        yielding $G(x) =  e^{x t_x} \le e^{1-x}$ for all $x> 1$.
        Taking $\gamma=1/x$ proves the desired inequality. 
        \item[(ii)] The first inequality follows by checking $E$ following an exponential distribution with mean $1$.
        The relations $\mathcal{E}_{\rm LCD} \subseteq \mathcal{E}_{\rm U}$
        and 
        $\mathcal{E}_{\rm LCD} \subseteq \mathcal{E}_{\rm LCS}$  give the second inequality. 
        The last inequality is shown in part (i) of Theorem \ref{thm:log-distribution}.
        \item[(iii)] Let $m>1$, $b\in (0,1)$ and $\lambda >0$. 
        Consider the cumulative distribution function $F$ given by 
        $$F(x) = (1-b)+b e^{-\lambda (m-x)_+} \mbox{ for $x\ge 0$}.$$
        It is straightforward to verify that the distribution function $F$ is log-concave,
        and its survival function is given by $x \mapsto b (1 - e ^{-\lambda (m-x) })$ on $ [0,m]$. 
        We further choose the parameters so that $F$ has mean $1$, that is, 
        $$
        1=\int_0^m b (1- e ^{-\lambda (m-x) }) \d x = b \left( m  - \frac{1-e^{-\lambda m}}{\lambda }\right),
        $$
        and this is equivalent to $b= \lambda /(m\lambda -1+e^{-\lambda m})$.  For any fixed $m>1$, if we send $\lambda \to \infty$, then $b\to 1/m\in (0,1)$. 
        Therefore, for any fixed $m$ we can take $\lambda $ large enough so that $F$ is the distribution function of some $E\in \mathcal E_{\rm LCF}$. 
        Hence, for $x\in [0,m),$
        $$\p(E>x) =  \frac{\lambda (1 - e ^{-\lambda (m-x) }) }{ m\lambda -1+e^{-\lambda m}}.$$
        Sending  $\lambda \to \infty$ we get
        $\p(E>x) \to 1/m$. 
        This implies that, for any $x \in [1,m)$, we have 
        $R_{1/x}(\mathcal E_{\rm LCF} ) \ge 1/m$.
        Since $m$ is arbitrary, this leads to 
        $R_{\gamma}(\mathcal E_{\rm LCF} ) \ge \gamma$ for $\gamma \in (0,1]$. Together with Markov's inequality, implying $R_{\gamma}(\mathcal E_{\rm LCF} ) \le \gamma$, we get the desired statement.
    \end{enumerate}

\section{Thresholds for combined e-values}\label{app:combined}

In Section \ref{sec:combined}, we investigated improved thresholds for stopped e-processes. In this section, we will investigate the related problem of finding improved thresholds in situations where multiple e-values are combined in a fixed-sample setting. 

\subsection{Thresholds for products of e-values}

We will now investigate a method to improve the threshold of products of independent or sequential e-variables in a fixed sample size setting. We first define the multiplicative strong unimodality (MSU) property proposed by \cite{cuculescu1998multiplicative}. A random variable $X$ is said to have the MSU property if, for any random variable $Y$ with a unimodal density function, the random variable $XY$ has a unimodal density function. Denote by
$$\mathcal E_{\rm MSU} = \{E\in\mathcal E_0: \mbox{$E$ is multiplicative strongly unimodal over $[0, \infty)$}\}.$$ Proposition 3.6 of \cite{cuculescu1998multiplicative} states that if $X$ is a unimodal random variable with a mode at zero, then the product of $X$ with any other random variable defined on $\mathbb{R}$ will also have a mode at zero; it follows that any e-variable with a decreasing density on $[0, \infty)$ will have the MSU property. Note that this does not coincide with the class $\mathcal{E}_{\rm D}$ since we require the mode to be zero, while the left end-point of supports for e-variables in $\mathcal{E}_{\rm D}$ could be larger than 0. Theorem 3.7 of \cite{cuculescu1998multiplicative} provides the condition for which e-variables will have the MSU property. Letting $f$ denote the density function of a random variable in $\mathcal E_{\rm MSU}$, then 
\begin{itemize}
    \item[(i)] $f$ has a decreasing density on $[0, \infty)$, implying it is unimodal with a mode at zero or
    \item[(ii)] $x \mapsto f(e^x)$ is log-concave on $[0, \infty)$.
\end{itemize}
Distributions that are log-concave with respect to $x \mapsto f(e^x)$ include beta, Cauchy, Fisher, gamma, inverse Gaussian, logistic, log-normal, Gaussian, Student, Pareto distributions, along with their location-scale transforms; see \cite{alimohammadi2016convolutions} for more examples. 

The following is immediately apparent.  
\begin{proposition}\label{prop:e-process}
    Let $S_T=\prod_{t=1}^T E_t$. 
    \begin{enumerate}
        \item[(i)] Let $E_t$ be sequential e-values for $t \in [T - 1]$ and let $E_T$ be independent of $S_{T-1}$ and have a decreasing density on $[0, \infty)$. Then, $S_T \in \mathcal{E}_{\rm D}$. 
        \item[(ii)] Let $E_t \in \mathcal{E}_{\rm MSU}$ for $t \in [T]$ be independent e-variables. Then, $S_T \in \mathcal{E}_{\rm U}$. 
    \end{enumerate}
\end{proposition}

Proposition \ref{prop:e-process} (i) tells us that, as long as we have one e-variable with a decreasing density on $[0, \infty)$ that is independent of other e-variables, then we can take the product and boost the threshold by 2.

When constructing e-processes based on betting strategies, one often uses $1-\lambda_t + \lambda_t E_t$ instead of $ E_t$, for some $\lambda_t$, possibly depending on previous e-values, which takes values in $[0,1]$; see \cite{S21} and \cite{WR24}. For $\mathcal E$ being any of $ \mathcal E_{\rm D}$, $ \mathcal E_{\rm U}$, $\mathcal E_{\rm LCD}$ and $\mathcal E_{\rm LCF}$, we have the convenient property that if $E\in \mathcal E$, then $1-\lambda + \lambda E\in \mathcal E$ for all constant $\lambda \in [0,1]$. This holds because properties such as monotonicity, unimodality, and log-concavity are invariant to changes in location and scale. However, part (i) of Proposition \ref{prop:e-process} requires that $E_T$ is decreasing on $[0, \infty)$, so it is not valid to apply it to $1-\lambda+\lambda E$ in this case.

\subsection{Weighted average of independent e-values}

We now consider the weighted average of independent e-variables. Note that a weighted average of e-variables is an e-variable regardless of the dependence structure, and using the weighted average is the only admissible way of merging arbitrarily dependent e-values, as shown by \cite{VW21, W25}. However, the following result relies on the independence assumption for an improved threshold. 

\begin{proposition}\label{prop:chernoff-bound-lcd}
        Let $E_t \in \mathcal{E}_{\rm LCD}$ for $t \in [T]$ be independent e-variables, $(w_1,\dots,w_T)$ is in the standard simplex, and $M_T = \sum_{t = 1}^{T} w_t E_t$. Denote by $w_{(1)} = \min_t w_t$. For $\gamma \in (0, 1]$, 
 $$\p(M_T\ge 1/\gamma)\leq \left(\gamma^{-w_{(1)}}e^{-w_{(1)}(1/\gamma - 1)}\right) \wedge R_\gamma(\mathcal{E}_{\rm U}).$$
\end{proposition}

\begin{proof}
From Theorem 1.6 of \cite{marsiglietti2022moments}, we have
$\p(M_T \geq t) \leq e^{-w_{(1)}(t - 1 - \log t)}$ for all $t \geq 1$. Replacing $t = 1/\gamma$ for $\gamma \in (0, 1]$ completes the first bound. If $E \in \mathcal{E}_{\rm LCD}$ and $w$ is a non-zero real constant, then $wE \in \mathcal{E}_{\rm LCD}$. Further, since the convolution of log-concave density functions is unimodal (see, for instance, \cite{ibragimov1956composition}), we have $M_t \in \mathcal{E}_{\rm U}$, showing the second bound. 
\end{proof}


We provide in Table \ref{tab:improved-threshold-mean} some improved thresholds for the average of independent e-variables with log-concave densities (using equal weights for each e-variable). We provide thresholds based on Proposition \ref{prop:chernoff-bound-lcd}, but note that for $T = 1$, the bound in Part (ii) of Theorem \ref{th:log-transformed} holds and gives a smaller threshold in Table \ref{tab:improved-threshold}. The exponential bound provides the lower threshold for small values of $T$ and $\alpha$, while the unimodal bound takes over in other situations. 

\begin{table}[ht]
	\caption{Improved thresholds for the average of $T$ e-variables with log-concave densities.}\label{tab:improved-threshold-mean}
	\centering
	\begin{tabular}{rrrrrrr}
	    &        \multicolumn{6}{c}{$\alpha$}        \\
	 $T$   &  0.001 &  0.01 &  0.02 &  0.05 & 0.1 & 0.2 \\ 
        \cmidrule(lr){1-1}\cmidrule(lr){2-7}
		 1 &  10.23 &  7.64 &  6.83 &  5.74 & 4.9 & 2.5 \\
		 2 &  17.69 & 12.76 & 11.24 &  9.21 & 5.0 & 2.5 \\
		 5 &  39.21 & 27.33 & 23.73 & 10.00 & 5.0 & 2.5 \\
	   10 &  74.39 & 50.00 & 25.00 & 10.00 & 5.0 & 2.5 \\
	   20 & 144.13 & 50.00 & 25.00 & 10.00 & 5.0 & 2.5
	\end{tabular}
\end{table}

\section{Proofs for results in Section \ref{sec:boosting}}\label{app:boosting-proof}

In this section, we provide the proof of various results in Section \ref{sec:boosting} along with additional remarks.

\begin{proof}[Proof of Proposition \ref{prop:e-bh-decreasing}]
    	It suffices to find a sequence of random variables within $\mathcal{E}_{\rm D}$ to satisfy \eqref{eq:enhance-3} with equality, hence no boosting is possible (that is, $B^{\rm AD}_\alpha\left(\mathcal{E}_{\rm D}\right) = 1$). For any $w \in (0, 2\alpha)$, let $E$ be a mixture random variable of a uniform distribution on $[0,2/w]$ with weight $w$ and a point mass at $0$. Then, we have
	$$\E\left[\alpha E \id_{\{\alpha E \ge 1\}}\right] = \int_{1/\alpha}^{2/w} \alpha x \frac{w^2}{2} \d x = \alpha - \frac{w^2}{4\alpha}.$$
	Since we can make $w$ arbitrarily small, we have that $B^{\rm AD}_\alpha (\mathcal E_{\rm D}) = 1$.
\end{proof}

\begin{remark}
	Consider again the e-variable defined in the proof of Proposition \ref{prop:e-bh-decreasing}. If $w < \alpha / K$, then we have for any $b \geq 1$ that
	$$\E[T(\alpha b E)] \leq K w \leq \alpha,$$
	so, in this case, it is possible to boost the e-values by an arbitrary amount while still satisfying the criterion, but we could not boost the e-values at all using the simplified problem that does not rely on $K$.
\end{remark}

We will require the following lemma for the proof of Proposition \ref{th:ad-ld}.

\begin{lemma}\label{lemma:log-concave-stochastic-order}
    For  $E \in \mathcal{E}_{\rm LCS}$ and $t>0$, we have  $\E[E\vert E > t] \leq \E[E] + t.$
\end{lemma}

\begin{proof}
    Since $E \in \mathcal{E}_{\rm LCS}$, we have from Theorem 3.B.19 (b) of \cite{SS07} that $(E - t \vert E > t)$ is decreasing in $t$ in the sense of the dispersive order. It follows from Theorem 3.B.13 (a) of \cite{SS07} that $(E - t \vert E > t)$ is smaller than $E$ under the usual stochastic order. Therefore, 
    $\E[E - t \vert E > t] \leq \E[E] \leq 1$ and $\E[E\vert E > t] \leq 1 + t.$
\end{proof}


\begin{proof}[Proof of Proposition \ref{th:ad-ld}]
Notice first that 
        \begin{equation}\label{eq:ad-boosting-decomposition}
            \E\left[\alpha b E \id_{\{\alpha b E \geq 1\}}\right] = \alpha b \p\left(E \geq \frac{1}{\alpha b}\right)\E\left[E \left\vert E \geq \frac{1}{\alpha b}\right.\right].
        \end{equation}
        We start with the upper bound. If $E$ is exponentially distributed with mean 1, \eqref{eq:ad-boosting-decomposition} becomes
        $$\E\left[\alpha b E \id_{\{\alpha b E \geq 1\}}\right] = \alpha b \exp\{-1/(\alpha b)\}\left(1 + \frac{1}{\alpha b}\right).$$ It follows that the problem in \eqref{eq:enhance-3} becomes 
        $$\exp\{-1/(\alpha b)\}\left(\alpha b + 1\right) = \alpha,$$
        and solving for $b$ leads to $c^{\rm AD}_2(\alpha)$, which gives an attainable boosting factor for a random variable in $\mathcal{E}_{\rm LS}$ and proves the upper bound for $B^{\rm AD}_\alpha\left(\mathcal{E}_{\rm LCS}\right)$. 
        
        
        For the lower bound, let us first find an upper bound for the quantity in \eqref{eq:ad-boosting-decomposition}. We may bound the probability term in \eqref{eq:ad-boosting-decomposition} by the bound in part (i) of Theorem \ref{thm:log-distribution} and bound the conditional expectation term by its upper bound of $1 + 1/(\alpha b)$; the latter bound follows from Lemma \ref{lemma:log-concave-stochastic-order}. This gives us
        $$\E\left[\alpha b E \id_{\{\alpha b E \geq 1\}}\right]\leq \alpha b \exp\left\{1 - \frac{1}{\alpha b}\right\}\left(1 + \frac{1}{\alpha b}\right).$$
        Solving for $b$ in 
        $$\alpha b \exp\left\{1 - \frac{1}{\alpha b}\right\}\left(1 + \frac{1}{\alpha b}\right) = \alpha$$
        proves the lower bound for $B^{\rm AD}_\alpha\left(\mathcal{E}_{\rm LCS}\right)$.
    \end{proof}

     \begin{remark}
	For any $\alpha \in (0, 1)$, we may obtain an analytical lower bound for $c^{\rm AD}_1(\alpha)$ and upper bound for $c^{\rm AD}_2(\alpha)$ that are slightly looser than the bounds in Proposition \ref{th:ad-ld}. Note that since 
		$$\exp\{-1/(\alpha b)\}\left(\alpha + 1\right) \leq \exp\{-1/(\alpha b)\}\left(\alpha b + 1\right),$$
		we may obtain an upper bound for $c^{\rm AD}_2(\alpha)$ by solving for $b$ in  $\exp\{-1/(\alpha b)\}\left(\alpha + 1\right) = \alpha,$  such that
		\begin{equation}\label{eq:c2_ad_upper_bound}
		    c^{\rm AD}_2(\alpha) \leq \frac{1}{\alpha \log \left(\frac{1 + \alpha}{\alpha}\right)} \leq -\frac{1}{\alpha \log \alpha}.
		\end{equation}
		Then, we have that $c^{\rm AD}_1(\alpha)$ is the solution for $b \geq 1$ in 
		\begin{equation}\label{eq:c1_lower_bound_1}
			\exp\left\{-1/(\alpha b)\right\}(1 + \alpha b) = \alpha \exp\{-1\}.
		\end{equation}
  We may obtain a lower bound for $c^{\rm AD}_1(\alpha)$ by replacing the second $b$ in \eqref{eq:c1_lower_bound_1} by its upper bound \eqref{eq:c2_ad_upper_bound}, hence a lower bound for $c^{\rm AD}_1(\alpha)$ is the solution for $b \geq 1$ in
 $\exp\left\{-1/(\alpha b)\right\}(1 - 1/\log (\alpha)) = \alpha \exp\{-1\},$ 
  which leads to 
  $$ \left(\alpha \left[\log \alpha - 1 + \log(-\log(\alpha)) - \log(1 - \log(\alpha))\right]\right)^{-1} \leq c_1^{\rm AD}(\alpha).$$
  Finally, let us remark that the bounds in Proposition \ref{th:ad-ld} are asymptotically tight as $\alpha \to 0$ since
  \begin{align*}
      \lim_{\alpha \to 0^+} \frac{c_2^{\rm AD}(\alpha)}{c_1^{\rm AD}(\alpha)} &= \lim_{\alpha \to 0^+} \frac{\log \alpha - 1 + \log(-\log(\alpha)) - \log(1 - \log(\alpha))}{\log \alpha} = 1.
  \end{align*}
\end{remark}

\begin{proof}[Proof of Proposition \ref{th:prds-ld}]
    	We have for any $x \ge 1/(\alpha b)$ that
        $$x\p(\alpha b E \geq x) = x\p\left(E \geq \frac{x}{\alpha b}\right) \leq x \exp \left\{1 - \frac{x}{\alpha b}\right\},$$
        where the inequality is due to part (i) of Theorem \ref{thm:log-distribution}. Note that $x \mapsto x\exp \left\{1 - \frac{x}{\alpha b}\right\}$ is unimodal with a mode at $\alpha b$, so it is a decreasing function of $x$ whenever $x \ge \alpha b$. Suppose that $1 \ge \alpha b$, then we have
        $$\max_{x \geq 1} x\p(\alpha b E \geq x) 
             \leq \max_{x \ge 1}x \exp \left\{1 - \frac{x}{\alpha b}\right\} \le \exp \left\{1 - \frac{1}{\alpha b}\right\} \leq \alpha.$$
        Solving the equation on the right-hand side with equality leads to $c^{\rm PR}_1(\alpha)$. One may verify that $1 \ge \alpha c_1^{\rm AD}(\alpha)$ whenever $\alpha \in (0, 1]$. Suppose that $1 \le \alpha b$, then 
     $$\max_{x \geq 1} x\p(\alpha b E \geq x) 
             \leq \max_{x \ge 1}x \exp \left\{1 - \frac{x}{\alpha b}\right\} \le \alpha b\exp \left\{1 - \frac{\alpha b}{\alpha b}\right\} = \alpha b.$$
     We, therefore, require $b = 1$ to satisfy \eqref{eq:enhance-4}. Note that we don't have $1 \ge \alpha b$ for $\alpha \in (0, 1]$, so a lower bound of 
     $B_\alpha^{\rm AD}(\mathcal{E}_{\rm LCS})$ is $c^{\rm PR}_1(\alpha)$ for $\alpha \in (0, 1]$. As for the upper bound, finding an attainable boosting factor for a random variable $E \in \mathcal{E}_{\rm LCS}$ suffices. Letting $E$ be exponentially distributed with mean 1, we must find $b$ such that 
     \begin{equation}\label{eq:b_sup_1}
         \max_{x \geq 1} x\p(\alpha b E \geq x) = \max_{x \geq 1} x \exp\left\{-x/(\alpha b)\right\} = \alpha.
     \end{equation}
     The function $x \mapsto x\exp \left\{- x/(\alpha b)\right\}$ is unimodal, where the mode is located at $\alpha b$. When $1 \le \alpha b$, the solution to \eqref{eq:b_sup_1} is $-1/(\alpha\log(\alpha))$, which is valid when $1 \le -\alpha /(\alpha\log(\alpha))$, or equivalently, when $\alpha \le \exp\{-1\}$. When $1 \le \alpha b$, the solution to \eqref{eq:b_sup_1} is $b = \exp\{1\}$; this occurs when $\alpha > \exp\{-1\}$.
    \end{proof}

\section{Additional simulation results}\label{app:more-simulations}

\subsection{Tests within the Gaussian family, continued}\label{ss:more-simulations-gaussian}

We continue the setting in Section \ref{ss:simulation-gaussian}. In Figure \ref{fig:e-process}, we present a simulated path for the process $\{S_n\}_{n\in [N]}$ (with its expected value $0.3\times n$) along with the e-processes and the supremum process for this example with $N = 1000$. One observes that for this sample path, the process $\{S_n\}_{n\in [N]}$ starts above its mean but reverts to it for the second half of the process. Hence, for this simulated path, the supremum process for $n$ under 300 is close to the e-process with $\mu = 0.4$ (since the empirical average is closer to 0.4 than 0.3). As the $\{S_n\}_{n\in [N]}$ approaches its mean for $n$ above 500, the supremum process matches the e-process with $\mu = 0.3$.

\begin{figure}[ht]
	\centering

        \includegraphics[width=\textwidth]{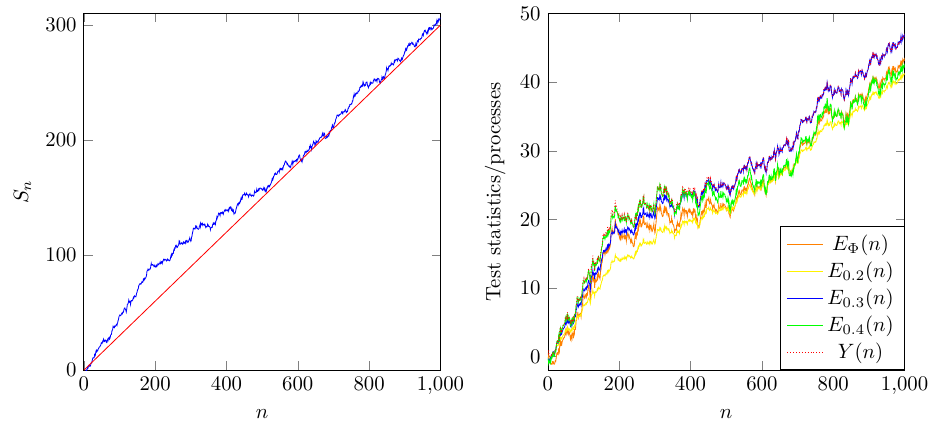}

	\caption{Left: Simulated path for one run of $S_n$; Right: E-processes and supremum process for the same run of $S_n$.}\label{fig:e-process}
\end{figure}

In Figure \ref{fig:e-process-average}, we present the average e-process and supremum process for 100 replications (where the average is taken over the log of the e-process and the supremum process). For the e-processes with the likelihood ratio, the average of the log is close to the expected value of the log e-process, which is $\delta n (\mu - \delta / 2)$. The mixture e-process becomes larger than the likelihood ratio e-processes with a misspecified alternative for $n$ larger than 600. The growth optimal e-process is the likelihood ratio e-process with the correct alternative, that is, $\{E_{\delta}(n)\}_{n\in [N]}$ (see, for instance, \cite{GDK24, S21}). The supremum process is larger than the growth-optimal e-process, but it does not lead to an anytime-valid test.

\begin{figure}[ht]
	\centering

        \includegraphics[width = 0.5\textwidth]{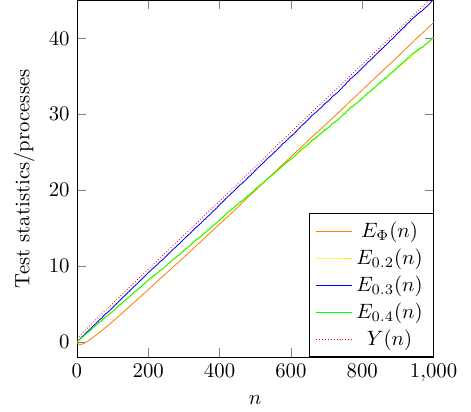}
    
	\caption{Average of the log e-process and supremum process.}\label{fig:e-process-average}
\end{figure}

We showed in Section \ref{ss:simulation-gaussian} that improved thresholds and taking the supremum of e-values lead to power improvements. Let us now validate that the tests control the type-I errors through simulations. In Table \ref{tab:rejection-rates-null}, we present simulated rejection rates when data come from the null $\mathrm{N}(0, 1)$ using the tests from Section \ref{ss:simulation-gaussian}. The type-I error is controlled at 0.05 for all tests. The type-I error is close to the control level for optional stopping strategies as $n$ increases, while the type-I error for the likelihood ratio e-processes decreases to 0 as $n$ increases. The test based on the supremum process using the improved threshold of $\mathcal{T}_{\alpha}(\mathcal{E}_{\rm LN})$ achieves a simulated type-I error equal to the target level, implying that it is not wasteful at this level.

\begin{table}[t]
	\centering
	\caption{Simulated rejection rates when testing within the Gaussian family when data come from the null $\mathrm{N}(0, 1)$, with a type-I error control of $\alpha=0.05$.}\label{tab:rejection-rates-null}
\begin{tabular}{rrrrrr}
                       &                                               &  \multicolumn{4}{c}{$n$}  \\
                  Test &                                     Threshold &   10 &   50 &  100 &  500 \\\hline
                       &       $\mathcal{T}_{\alpha}(\mathcal{E}_{0})$ & 0.00 & 0.00 & 0.01 & 0.00 \\
            Likelihood &   $\mathcal{T}_{\alpha}(\mathcal{E}_{\rm U})$ & 0.00 & 0.01 & 0.02 & 0.00 \\
                 ratio & $\mathcal{T}_{\alpha}(\mathcal{E}_{\rm LUS})$ & 0.00 & 0.02 & 0.02 & 0.00 \\
           $\mu = 0.2$ &  $\mathcal{T}_{\alpha}(\mathcal{E}_{\rm LN})$ & 0.01 & 0.05 & 0.05 & 0.01 \\
                       &    OS+$\mathcal{T}_{\alpha}(\mathcal{E}_{0})$ & 0.00 & 0.01 & 0.02 & 0.04 \\ \hline
                       &       $\mathcal{T}_{\alpha}(\mathcal{E}_{0})$ & 0.00 & 0.01 & 0.01 & 0.00 \\
            Likelihood &   $\mathcal{T}_{\alpha}(\mathcal{E}_{\rm U})$ & 0.00 & 0.02 & 0.01 & 0.00 \\
                 ratio & $\mathcal{T}_{\alpha}(\mathcal{E}_{\rm LUS})$ & 0.01 & 0.02 & 0.01 & 0.00 \\
           $\mu = 0.3$ &  $\mathcal{T}_{\alpha}(\mathcal{E}_{\rm LN})$ & 0.03 & 0.04 & 0.02 & 0.00 \\
                       &    OS+$\mathcal{T}_{\alpha}(\mathcal{E}_{0})$ & 0.00 & 0.02 & 0.03 & 0.04 \\ \hline
                       &       $\mathcal{T}_{\alpha}(\mathcal{E}_{0})$ & 0.00 & 0.01 & 0.00 & 0.00 \\
            Likelihood &   $\mathcal{T}_{\alpha}(\mathcal{E}_{\rm U})$ & 0.01 & 0.01 & 0.01 & 0.00 \\
                 ratio & $\mathcal{T}_{\alpha}(\mathcal{E}_{\rm LUS})$ & 0.01 & 0.02 & 0.01 & 0.00 \\
           $\mu = 0.4$ &  $\mathcal{T}_{\alpha}(\mathcal{E}_{\rm LN})$ & 0.04 & 0.03 & 0.01 & 0.00 \\
                       &    OS+$\mathcal{T}_{\alpha}(\mathcal{E}_{0})$ & 0.00 & 0.03 & 0.04 & 0.04 \\ \hline
                       &       $\mathcal{T}_{\alpha}(\mathcal{E}_{0})$ & 0.01 & 0.01 & 0.01 & 0.01 \\
              Supremum &   $\mathcal{T}_{\alpha}(\mathcal{E}_{\rm U})$ & 0.02 & 0.02 & 0.02 & 0.02 \\
           comonotonic & $\mathcal{T}_{\alpha}(\mathcal{E}_{\rm LUS})$ & 0.02 & 0.02 & 0.02 & 0.02 \\
                       &  $\mathcal{T}_{\alpha}(\mathcal{E}_{\rm LN})$ & 0.05 & 0.05 & 0.05 & 0.05 \\ \hline
	\multirow{3}{*}{Mixture} &       $\mathcal{T}_{\alpha}(\mathcal{E}_{0})$ & 0.00 & 0.00 & 0.00 & 0.00 \\
                       &   $\mathcal{T}_{\alpha}(\mathcal{E}_{\rm U})$ & 0.01 & 0.00 & 0.00 & 0.00 \\
                       &    OS+$\mathcal{T}_{\alpha}(\mathcal{E}_{0})$ & 0.01 & 0.02 & 0.03 & 0.03
\end{tabular}
\end{table}

\subsection{Tests within the gamma family}\label{ss:gamma-como}

We now illustrate the technique derived in Section \ref{sec:43} for obtaining maximum likelihood e-variables. Suppose we want to test whether two gamma distributions differ, using the scale and rate parameterization. The parameter vector is ${\theta} = (\alpha, \beta)$ and the natural parameter (in the context of exponential families) is ${\eta}({\theta}) = (\alpha - 1, -\beta)$.

For a null hypothesis 
$\alpha = \alpha_0$ and $\beta = \beta_0$, let us define the sets 
\begin{enumerate}[(a)]
    \item $\Theta_1 = \{(\alpha, \beta) : \alpha > 0;~ \beta > 0\}$;
    \item $\Theta_2 = \{(\alpha, \beta) : \alpha_0 < \alpha;~0 < \beta < \beta_0\}$;
    \item $\Theta_3 = \{(\alpha, \beta) : 0 < \alpha < \alpha_0;~ \beta_0 < \beta\}$.
\end{enumerate}

As mentioned in Section \ref{sec:43}, since ${\eta}({\theta}) - {\eta}({\theta}_0)$ have the same sign for the choices $\Theta' = \Theta_2$ or $\Theta' = \Theta_3$, the collection of likelihood ratio e-variables
$\{E_{\theta}(n) : \theta \in \Theta'\}$ is comonotonic.
Letting $(\hat{\alpha}, \hat{\beta})$ denote the maximum likelihood estimate of $(\alpha, \beta)$ over $\Theta_2$ or $\Theta_3$ based on data $(x_1, \dots, x_n)$, we can use the statistic
\begin{equation}\label{eq:como_test_gamma}
    Y(n) := \prod_{i = 1}^n \frac{\Gamma(\alpha_0)}{\Gamma(\hat{\alpha})}
\frac{\hat{\beta}^{\hat{\alpha}}}{\beta_0^{\alpha_0}}
x_i^{\hat{\alpha} - \alpha_0}
\exp\left\{-(\hat{\beta} - \beta_0)x_i\right\}
\end{equation}
with rejection threshold $1/\alpha$ to construct a valid test, as discussed in Section \ref{sec:43}. On the other hand, if $\Theta' = \Theta_1$, the collection of e-variables is no longer comonotonic and taking the supremum may not guarantee type-I error control. 

Let us validate through simulations that the test based on the supremum of comonotonic e-variables controls the type-I error 
over different alternative hypotheses. We simulate observations from a gamma distribution with $\alpha_0 = 1$ and $\beta_0 = 1$ (corresponding to an exponentially distributed random variable with a mean of 1). We test using the test statistic in \eqref{eq:como_test_gamma} for $n$ varying between 2 and 50. 
\begin{figure}[ht]
	\centering

        \includegraphics{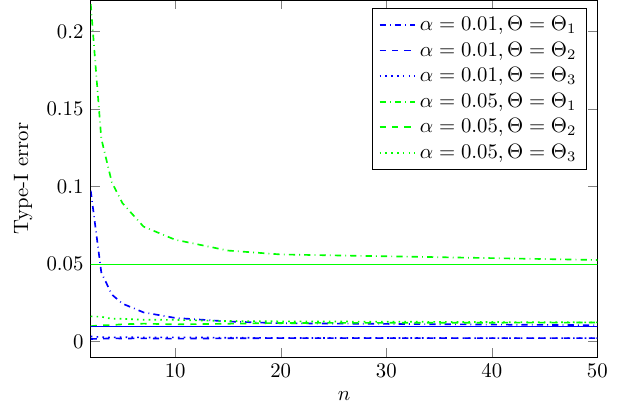}

	\caption{Simulated type-I errors for the gamma supremum test at levels 0.01 and 0.05. Full horizontal lines correspond to the type-I error control level. }\label{fig:sim-gamma}
\end{figure}
In Figure \ref{fig:sim-gamma}, we present simulated type-I errors at control levels 0.01 and 0.05 for 100,000 replications. When using the supremum test over $\Theta_2$ or $\Theta_3$, we obtain valid type-I error control no matter the number of observations, validating Proposition \ref{prop:sup-comonotonic}. In that case, the simulated type-I error is about four times smaller than the control level. On the other hand, the supremum test over $\Theta_1$ is not valid since the collection $\{E_{\theta}(n) : \theta \in \Theta_1\}$ is not comonotonic. Indeed, we observe that the rejection thresholds for the supremum test over $\Theta_1$ are above the control level for each value of $n$. 

Assume now that data come from a gamma distribution with parameters $\alpha = 1.1$ and $\beta = 0.9$, such that the random variable's mean is $11/9$. We use two test statistics to test the null hypothesis $\alpha_0 = \beta_0 = 1$. The first is the likelihood ratio e-variable using the true alternative that $\alpha = 1.1$ and $\beta = 0.9$. The second is against the composite alternative hypothesis that $(\alpha, \beta) \in \Theta_2 = \{(\alpha, \beta) : \alpha_0 < \alpha;~0 < \beta < \beta_0\}$ using the supremum test statistic. 
\begin{figure}[t]
	\centering
	\includegraphics{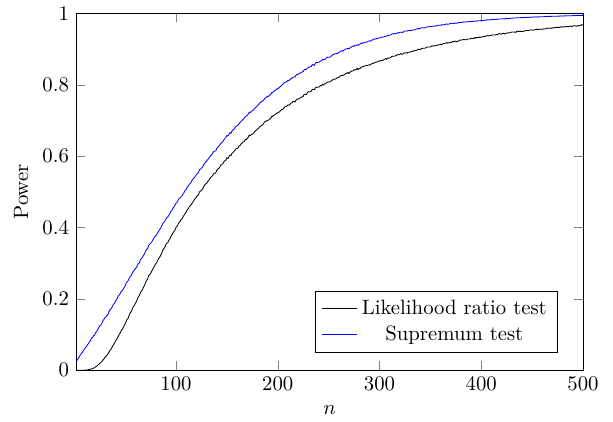}


	\caption{Simulated power when testing within the gamma family. The likelihood ratio e-test is against the true alternative, while the supremum test is against the alternative $\alpha > \alpha_0$ and $\beta_0 > \beta > 0$.}\label{fig:sim-gamma-power}
\end{figure}
We present power curves in Figure \ref{fig:sim-gamma-power}. Using the supremum test over the likelihood ratio test enhances power by a factor of 1.81 for $n = 50$, 1.17 for $n = 100$, 1.09 for $n = 200$ and 1.03 for $n = 500$. To achieve a power of 0.9, we require around 268 observations using the supremum test and 341 observations using the likelihood ratio e-test.

\subsection{Improved threshold for universal inference, continued}\label{app:ui-simulation}

In Section \ref{ss:ui-simulation}, we conducted a simulation study for universal inference e-tests when testing the null hypothesis ${\rm N}(0, 1)$ against the alternative that data come from a mixture of two Gaussians. The motivating example in the Introduction uses a different setup, resulting in the e-values presented in Figure \ref{fig:ui-e-values}. The difference between the setup of Section \ref{ss:ui-simulation} and in the Introduction is that the alternative hypothesis in the Introduction was  $ 0.5{\rm N}(\mu_1, 1) + 0.5{\rm N}(\mu_2, 1)$ for unknown $(\mu_1, \mu_2)$, that is, the weights and variances in the Gaussian mixture was assumed known. We repeat the simulation study for this setup for completeness. 

In Figure \ref{fig:ui-e-values}, the e-values have a decreasing density over $[1, \infty)$, while the log-transformed e-values have a decreasing density over their positive domain. It is, therefore, reasonable to assume that the e-variables belong to $\mathcal{E}_{\rm D>1}$ and $\mathcal{E}_{\rm LD>0}$. 

Figure \ref{fig:ui-power-curve-app} shows power curves when the data come from $0.5{\rm N}(-\mu, 1) + 0.5{\rm N}(\mu, 1)$, where $\mu$ is a signal parameter. For $\mu = 0.5$, 
using the threshold of $T_{0.1}(\mathcal{E}_{\rm D>1})$ leads to 26\% more power compared with $T_{0.1}(\mathcal{E}_{0})$. Similarly, using the threshold $T_{0.1}(\mathcal{E}_{\rm LD>0})$ gives 35\% more power compared with the Markov bound. 

\begin{figure}[t]
	\centering
    \includegraphics{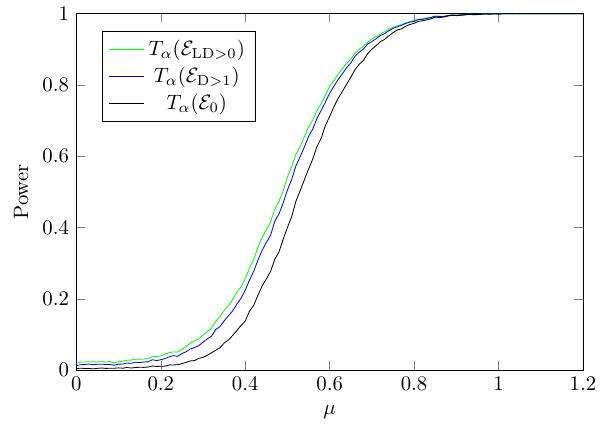}
			
   
                
                
	\caption{Power for universal inference e-tests with and without improved thresholds for the simplified Gaussian mixture model.}\label{fig:ui-power-curve-app}
\end{figure}

\subsection{The boosted e-BH, continued}\label{app:eBH}

To understand how well the boosted e-BH procedure performs, we compare the method with another e-BH procedure, but with the knowledge that the null e-variables are exponentially distributed with mean 1 (Table \ref{tab:e-BH-boosted-exact}). We also compare with the p-BH procedure, where we construct the p-values $p_k$ for the p-BH procedure as $\exp\{-e_k\}$ for $k \in \mathcal{K}$ (Table \ref{tab:p-BH-exact}). In this example, the PRDS property is not satisfied by the p-values $(p_1, \dots, p_K)$, so the p-BH procedure does not have a theoretically proven FDR control. Indeed, some numbers are slightly higher than $\alpha K_0/K$.

	\begin{table}[t]
		\centering
  \caption{Number of discoveries and realized FDP as percentages based on 1000 repetitions.}
    \begin{subtable}[b]{1\textwidth}
    \centering
		\caption{The boosted e-BH procedure with the knowledge of the null e-variable distribution}\label{tab:e-BH-boosted-exact}
		\begin{tabular}{rrrrrrr}
			              & \multicolumn{2}{c}{$b = 3$} & \multicolumn{2}{c}{$b = 4$} & \multicolumn{2}{c}{$b = 5$} \\ 			               \cmidrule(lr){2-3}\cmidrule(lr){4-5}\cmidrule(lr){6-7} 
			$\alpha$ (\%) & Discov. &          FDP (\%) & Discov. &          FDP (\%) & Discov. &          FDP (\%) \\ \hline
			            1 &       0 &                 0 &  157.99 &                 0 &  246.13 &                 0 \\
			            2 &    7.25 &                 0 &  200.71 &                 0 &  271.78 &                 0 \\
			            5 &  118.05 &                 0 &  252.11 &             0.001 &  309.00 &             0.003 \\
			           10 &  204.29 &             0.002 &  290.73 &             0.024 &  336.71 &             0.062 \\ \hline
		\end{tabular}
  \end{subtable}
\\\vspace{0.3cm}
  \begin{subtable}[b]{1\textwidth}
  \centering
		\caption{The p-BH procedure with  the knowledge of the null e-variable distribution}\label{tab:p-BH-exact} 
		\begin{tabular}{rrrrrrr}
			              & \multicolumn{2}{c}{$b = 3$} & \multicolumn{2}{c}{$b = 4$} & \multicolumn{2}{c}{$b = 5$} \\ 			 \cmidrule(lr){2-3}\cmidrule(lr){4-5}\cmidrule(lr){6-7} 
			$\alpha$ (\%) & Discov. &          FDP (\%) & Discov. &          FDP (\%) & Discov. &          FDP (\%) \\ \hline
			            1 &  298.58 &             0.503 &  340.52 &             0.488 &  368.48 &             0.490 \\
			            2 &  316.79 &             0.997 &  355.90 &             0.999 &  381.70 &             1.018 \\
			            5 &  346.01 &             2.562 &  381.84 &             2.479 &  405.95 &             2.533 \\
			           10 &  379.14 &             5.045 &  410.86 &             4.973 &  432.12 &             5.001 \\ \hline
		\end{tabular}
  
  \end{subtable}
	\end{table}

\end{appendix}

\end{document}